\documentclass[a4paper, 10pt, oneside, reqno]{amsart}

\usepackage[utf8]{inputenc}
\usepackage{amsfonts,amssymb,amsthm,amsmath}
\usepackage{mathrsfs}
\usepackage{geometry}
\usepackage{graphicx,graphics}
\usepackage{color}
\usepackage{tikz}
\usepackage{tikz-cd}
\usepackage{comment}
\usepackage{subfiles}
\usepackage{stmaryrd}
\usepackage{hyperref}
\usepackage{todonotes}

\hypersetup{colorlinks = true, linkbordercolor = {white}, linkcolor = {blue}, citecolor = {blue}}

\theoremstyle{plain}
\newtheorem{thm}{Theorem}[subsection]
\newtheorem{lem}[thm]{Lemma}
\newtheorem{prop}[thm]{Proposition}
\newtheorem{corol}[thm]{Corollary}
\newtheorem{conjecture}[thm]{Conjecture}
\theoremstyle{definition}
\newtheorem{defi}[thm]{Definition}

\theoremstyle{remark}
\newtheorem{rque}[thm]{Remark}

\newtheorem{ex}[thm]{Example}

\def \F {\mathbb{F}}

\def \Fq {\F_q}

\def \Fqb {\overline{\F}_q}
\def \Flb {\overline{\F}_{\ell}}

\def \Z {\mathbb{Z}}
\def \Zl {\Z_{\ell}}
\def \Zlb {\overline{\Z}_{\ell}}

\def \Q {\mathbb{Q}}

\def \Qlb {\overline{\Q}_{\ell}}

\def \Ind {\mathrm{Ind}}

\def \Gm {\mathbb{G}_m}

\def \Gbf {\mathbf{G}}
\def \Bbf {\mathbf{B}}
\def \Ubf {\mathbf{U}}
\def \Tbf {\mathbf{T}}
\def \Wbf {\mathbf{W}}
\def \Nbf {\mathbf{N}}
\def \BwB {\Bbf w\Bbf}

\def \Hom {\mathrm{Hom}}

\def \id {\mathrm{id}}

\def \1 {\mathbb{1}}

\def \End {\mathrm{End}}

\def \Frob {\mathrm{Frob}}
\def \Weil {\mathrm{Weil}}

\def \Frob {\mathrm{F}}
\def \Ad {\mathrm{Ad}}

\def \mon {\mathrm{mon}}

\def \triv {\mathrm{triv}}

\def \RGamma {\mathrm{R\Gamma}}

\def \DD {\mathrm{D}}
\def \pt {\mathrm{pt}}

\def \unip {\mathrm{unip}}
\def \Perv {\mathrm{Perv}}

\def \pt {\mathrm{pt}}

\def \eq {\mathrm{eq}}
\def \IC {\mathrm{IC}}

\def \Rep {\mathrm{Rep}}

\def \DDc {\DD_{\cons}}
\def \cons {\mathrm{c}}

\def \Perf {\mathrm{Perf}}

\def \inv {\mathrm{inv}}

\def \Irr {\mathrm{Irr}}

\def \Ocal {\mathcal{O}}

\def \tot {\mathrm{tot}}

\def \Irr {\mathrm{Irr}}

\def \Ocal {\mathcal{O}}

\def \Lcal {\mathcal{L}}

\def \tr {\mathrm{tr}}
\def \ch {\mathrm{ch}}

\def \DL {\mathrm{DL}}
\def \tilt {\mathrm{tilt}}
\def \maps {\mathrm{maps}}
\def \Ch {\mathrm{Ch}}
\def \ch {\mathrm{ch}}
\def \hc {\mathrm{hc}}
\def \mix {\mathrm{mix}}
\def \equi {\mathrm{eq}}
\def \ex {\mathrm{ex}}
\def \IH {\mathrm{IH}}

\newcommand{\Hh}{\mathcal{H}}

\title{Tilting representations of finite groups of Lie type}
\author{Arnaud Eteve}

\begin{document}

\maketitle

\noindent\textbf{Abstract.} Let $\Gbf$ be a connected reductive group over a finite field $\Fq$ of characteristic $p > 0$. In this paper, we study a category which we call Deligne--Lusztig category $\mathcal{O}$ and whose definition is similar to category $\mathcal{O}$. We use this to construct a collection of representations of $\Gbf(\Fq)$ which we call the tilting representations. They form a generating collection of integral projective representations of $\Gbf(\Fq)$. Finally we compute the character of these representations and relate their expression to previous calculations of Lusztig and we then use this to establish a conjecture of Dudas--Malle of 2014.

\tableofcontents

\section{Introduction}

Let $p > 0$ be a prime and $q = p^r$ a power of $p$ and let $\Fq \subset \Fqb = k$ be the finite field with $q$ elements and its algebraic closure. Let $\Gbf$ be a connected reductive group over $k$ equipped with a Frobenius endomorphism $\Frob : \Gbf \to \Gbf$, i.e. a purely inseparable isogeny such that a power of $\Frob$ is the Frobenius coming from an $\F_{q^s}$-structure. We fix $\Bbf = \Tbf\Ubf$ a Borel pair that is stable under $\Frob$ and we denote by $\Wbf$ the corresponding Weyl group. This paper has two main goals : 
\begin{enumerate}
\item Introduce and study a category which we call `Deligne--Lusztig category $\mathcal{O}$' whose definition mimics the geometric realization of classical category $\mathcal{O}$ as sheaves on $\Ubf \backslash \Gbf/\Bbf$. 
\item Introduce a class of representations of the finite group of Lie type ${\Gbf}^{\Frob}$, which we call tilting representations and study their properties.
\end{enumerate}

\subsection{Deligne--Lusztig category $\mathcal{O}$}

\subsubsection{Category $\mathcal{O}$}

Category $\mathcal{O}$ is a central object in representation theory, it links, through the Kazhdan--Lusztig conjectures \cite{KazhdanLusztig}, \cite{KLSchubert}, the theory of representations of semisimple Lie algebras with the geometry of the flag varieties and provides a categorification of the Hecke algebra associated to $\Wbf$, we refer to \cite{AcharBook} for a historical account. It is now standard to define 
$$\mathcal{O} = \DD^b_{c,\Ubf}(\Gbf/\Bbf, \Qlb)$$
as the category of $\ell$-adic sheaves, where $\ell \neq p$ is a prime number, on the flag variety of $\Gbf$ that are locally constant along $\Ubf$-orbits. Its perverse heart is an abelian category which has the structure of a highest weight category \cite{BGS} and therefore has very good homological properties. While it is known that this category provides a categorification of the Hecke algebra, it is not equipped with a monoidal structure. to reveal the algebra structure of the Hecke algebra, one must pass to Hecke categories. They come in two flavors.
\begin{enumerate}
\item The equivariant Hecke category: $\Hh^{\eq}$. We define it, for this introduction, as the category of sheaves
$$\DD^b_{c,\Bbf}(\Gbf/\Bbf, \Qlb) = \DD^b_c(\Bbf \backslash \Gbf/\Bbf, \Qlb)$$ 
the category of $\Bbf$-equivariant sheaves on $\Gbf/\Bbf$. This is a monoidal category that acts on category $\mathcal{O}$ which categorifies the Hecke algebra.
\item The monodromic Hecke category : $\Hh^{\mon}$, initially constructed by \cite{BezrukavnikovYun}. It is a certain category of sheaves on 
$\Ubf \backslash \Gbf/\Ubf.$
It follows from \emph{loc. cit.} that this category also provides a categorification of the Hecke algebra and acts on category $\mathcal{O}$. 
\end{enumerate}
Both of these actions categorify the left and right module structures of the Hecke algebra on itself. We equip $\Gbf$ with its Bruhat stratification $\Gbf = \sqcup_{w \in \Wbf} \Bbf w\Bbf$, this induces a stratification of $\Gbf/\Bbf$ and all the quotients mentioned before. We let $j_w : \BwB/\Bbf \subset \Gbf/\Bbf$ be the inclusion. Category $\mathcal{O}$ possesses four collections of perverse sheaves that are of interest to representation theory :
\begin{enumerate}
\item the standard sheaves $\Delta_w^{\mathcal{O}} = j_{w,!} \Qlb [\ell(w)]$, 
\item the costandard sheaves $\nabla_w^{\mathcal{O}} = j_{w,*} \Qlb[\ell(w)]$, 
\item the simple sheaves $\IC_{w}^{\Ocal} = j_{w,!*} \Qlb[\ell(w)]$, 
\item the indecomposable tilting sheaves $T_w^{\Ocal}$,
where $\ell(w)$ is the length of the element $w \in \Wbf$.
\end{enumerate}
The tilting objects are defined as follows. 
\begin{defi}
Let $A \in \Perv_{\Ubf}(\Gbf/\Bbf ,\Qlb)$. Then
\begin{enumerate}
\item a $\Delta$-filtration on $A$ is a filtration whose graded pieces are standard sheaves, 
\item a $\nabla$-filtration on $A$ is a filtration whose graded pieces are costandard sheaves, 
\item the object $A$ is tilting if it has both a $\Delta$ and a $\nabla$-filtration. 
\end{enumerate}
\end{defi}
In the context of category $\mathcal{O}$, it is known that there exists a unique indecomposable tilting sheaf $T_w^{\Ocal}$ such that $T_w^{\Ocal}$ contains $\Delta_w^{\Ocal}$ in some (equivalently any) $\Delta$-filtration with multiplicity one and $T_w^{\Ocal}$ is supported on the closure of $\BwB/\Bbf$. 

Hecke categories only have some of these collections, namely 
\begin{enumerate}
\item the equivariant Hecke category $\Hh^{\eq}$ has standard, costandard and simple objects but usually no tilting objects,
\item the monodromic Hecke category $\Hh^{\mon}$ has standard, costandard and tilting objects but usually not a good theory of simple object. 
\end{enumerate}

So far, we have restricted ourselves to the case of $\Qlb$-sheaves but much of what has been discussed so far holds with $\Flb$-sheaves instead, see for instance \cite{AcharRicheKoszulDuality1}, \cite{RicheSoergelWilliamson}. 

\subsubsection{Deligne--Lusztig catgory $\Ocal$}

In our context we also have the Frobenius endomorphism $\Frob : \Gbf \to \Gbf$. We are interested in the representation theory of the finite group ${\Gbf}^{\Frob}$. Consider the action of $\Gbf$ on itself by $g.x = gx\Frob(g^{-1})$, we denote this action by $\Ad_{\Frob}$. Since $\Gbf$ is connected Lang's theorem provides an isomorphism of algebraic stacks 
$$\frac{\Gbf}{\Ad_{\Frob}\Gbf} = \pt/{\Gbf}^{\Frob}$$ 
and by étale descent there is a natural equivalence of categories $\DD(\Rep_{\Lambda} {\Gbf}^{\Frob}) = \DD(\pt/{\Gbf}^{\Frob}, \Lambda)$, where $\Lambda$ is a coefficient ring for which the category of étale sheaves on $\pt/{\Gbf}^{\Frob}$ makes sense, see the notation section \ref{sec:Notations}. We introduce the stack 
$$\frac{\Ubf \backslash \Gbf/\Ubf}{\Ad_{\Frob}\Tbf}$$
which we call the horocycle stack. We had previously introduced this stack in \cite{EtevePaper1}. We define 
$$\Ocal^{\DL}_{\Lambda} = \DD(\frac{\Ubf \backslash \Gbf/\Ubf}{\Ad_{\Frob}\Tbf}, \Lambda)$$
the category of $\Lambda$-sheaves on it where $\Lambda \in \{\Flb, \Qlb, \Zlb\}$. We call this category `Deligne--Lusztig category $\mathcal{O}$'. We consider it as a variant of the classical category $\mathcal{O}$. To put some emphasis on the comparison, category $\Ocal$ is classically realized as the category of sheaves on 
$$(\Ubf \backslash \Gbf/\Ubf)/\Tbf$$
where the torus $\Tbf$ acts by right translations. 

We equip this stack with its Bruhat stratification and, for $w \in \Wbf$, the corresponding stratum is then isomorphic to 
$$\frac{\Ubf \backslash \BwB/\Ubf}{\Ad_{\Frob}\Tbf} = \pt/({\Tbf}^{w\Frob} \ltimes \Ubf \cap {^w}\Ubf).$$
We refer to section \ref{sec:DLCatO} for a discussion. There is a natural equivalence $\DD(\pt/({\Tbf}^{w\Frob} \ltimes \Ubf \cap {^w}\Ubf), \Lambda) = \DD(\Rep_{\Lambda} {\Tbf}^{w\Frob})$ which we normalize to be $t$-exact. It follows that $\Ocal^{\DL}$ is glued from the categories of representations of the finite groups ${\Tbf}^{w\Frob}$. Given a character $\chi : {\Tbf}^{w\Frob} \to \Lambda^{\times}$ of ${\Tbf}^{w\Frob}$, let $E_{\chi}$ be the projective cover of $\chi$ in the abelian category $\Rep_{\Lambda} {\Tbf}^{w\Frob}$ of representations of ${\Tbf}^{w\Frob}$ on $\Lambda$-modules. We set 
\begin{enumerate}
\item $\Delta_{w,\chi} = j_{w,!}E_{\chi}[\ell(w)]$, the standard objects 
\item $\nabla_{w,\chi} = j_{w,*}E_{\chi}[\ell(w)]$, the costandard objects.
\end{enumerate}
We denote by $\Ocal^{\DL, \heartsuit}$ the perverse heart of $\Ocal^{\DL}$.

We then proceed to study this category and we show that this category has much of the structure of a highest weight category. Let us summarize the main properties of this category, see Theorem \ref{thm:StructureCatO} and Lemma \ref{lem:CharDeltaFilt}.
\begin{thm}
\begin{enumerate}
\item The category $\Ocal^{\DL}$ is compactly generated, and the standard and costandard sheaves are compact. 
\item For all pairs $(w, \chi)$, where $w \in \Wbf$ and $\chi$ is a character of ${\Tbf}^{w\Frob}$, there exists a unique (up to isomorphism) indecomposable tilting sheaf $T_{w,\chi}$ supported on the closure of $\BwB$ containing $\Delta_{w,\chi}$ with multiplicity $1$ in a $\Delta$-filtration, moreover any indecomposable tilting sheaf is of this form. 
\item A sheaf $T \in \Ocal^{\DL, \heartsuit}$ is tilting if and only if for all $w$, the objects $j_w^*T$ and $j_w^!T$ are perverse and projective. 
\item The subcategory $\Ocal^{\DL}_{\tilt}$ of tilting sheaves is an additive, Karoubian category satisfying the Krull--Schmidt property and generates $\Ocal^{\DL}$. 
\end{enumerate}
\end{thm}

We also discuss the following relations 
\begin{enumerate}
\item We compare $\Ocal^{\DL}$ with the Hecke categories $\Hh^{\eq}$ and $\Hh^{\mon}$ and their twisted variants of \cite{LusztigYun} and \cite{Gouttard} in section \ref{sec:ComparisonHeckeCat}.
\item We define a mixed version $\Ocal^{\DL, \mix}_{\Qlb}$ of $\Ocal^{\DL}_{\Qlb}$ and we show categorification results in the same vein as the categorification results of $\Ocal$, see section \ref{sec:MixedCategoryO}. 
\item We compare the decomposition numbers of tilting sheaves of $\Ocal^{\DL}$ with the ones of $\Ocal$ in section \ref{sec:ComparisonCatO}.
\end{enumerate}
It follows from this discussion that the structure of $\Ocal^{\DL}$ is essentially controlled by Kazhdan--Lusztig type combinatorics.

\subsection{Tilting representations}

\subsubsection{Deligne--Lusztig theory}

Classically, the representation theory of ${\Gbf}^{\Frob}$ is studied using Deligne--Lusztig theory \cite{DeligneLusztig}. One first defines two collections of varieties 
\begin{enumerate}
\item given $w \in \Wbf$, we define
$$X(w) = \{g\Bbf, g^{-1}\Frob(g) \in \BwB\} \subset \Gbf/\Bbf$$
\item and given a lift $\dot{w} \in \Nbf(\Tbf)$ of the element $w \in \Wbf$, we define
$$Y(\dot{w}) = \{g\Ubf, g^{-1}\Frob(g) \in \Ubf\dot{w}\Ubf \} \subset \Gbf/\Ubf.$$ 
\end{enumerate}
These varieties are smooth of dimension $\ell(w)$ and are equipped with an action of the finite group ${\Gbf}^{\Frob}$. Taking their cohomology yields representations of ${\Gbf}^{\Frob}$ and one usually denotes by 
$$R_w^1 = \sum_i (-1)^{i}[H^i_c(X(w), \Qlb)]$$
the character of $\RGamma_c(X(w), \Qlb)$. More generally, since the variety $Y(\dot{w})$ also carries an action of ${\Tbf}^{w\Frob}$, we also define 
$$R_w^{\theta} = \sum_i (-1)^i[H^i_c(Y(\dot{w}), \Qlb)]_{\theta}$$ 
by taking $\theta$ isotypic parts, where $\theta$ is a character of ${\Tbf}^{w\Frob}$. A very important set of class functions on ${\Gbf}^{\Frob}$ are the uniform class functions. These are the class functions of ${\Gbf}^{\Frob}$ that are linear combinations of the characters $R_w^{\theta}$. An interesting example is provided by the character 
$$[\IH(\overline{X(w)}, \Qlb)] = \sum_i (-1)^i [\IH^i(\overline{X(w)}, \Qlb)]$$
of the intersection cohomology of the closure of the variety $X(w)$ in $\Gbf/\Bbf$ (which is usually a singular variety). There is an explicit relation between this character and the character $R_w^1$ expressed in terms of Kazhdan--Lusztig polynomials done in \cite{LusztigBook}, this relation is a reflection of some geometric relation between Deligne--Lusztig theory and the Hecke category. 

\subsubsection{Tilting representations}

There is a natural functor 
$$\ch =\alpha_!\beta^*[\dim \Ubf] : \Ocal^{\DL} \to \DD(\pt/{\Gbf}^{\Frob}, \Lambda)$$ 
where the maps $\alpha$ and $\beta$ are given by the following correspondence 
$$\frac{\Gbf}{\Ad_{\Frob}\Gbf} \xleftarrow{\alpha} \frac{\Gbf}{\Ad_{\Frob}\Bbf}\xrightarrow{\beta} \frac{\Ubf \backslash \Gbf/\Ubf}{\Ad_{\Frob}\Tbf}.$$
We refer to section \ref{sec:TiltingRep}.
This correspondence is a variant of the correspondence used by Lusztig to define character sheaves \cite{CharacterSheaves1} and more recently in \cite{LusztigCatCenter}. A variant with $\Bbf\backslash \Gbf/\Bbf$ instead of $\frac{\Ubf \backslash \Gbf/\Ubf}{\Ad_{\Frob}\Tbf}$ was studied in \cite{BonnafeDudasRouquier}. 

\begin{thm}[\cite{EtevePaper1}]
There are isomorphisms up to shifts 
$$\ch(\Delta_{w,\chi}) = e_{\chi}\RGamma_c(Y(\dot{w}), \Lambda)[\ell(w)]$$ 
and
$$\ch(\nabla_{w,\chi}) = e_{\chi}\RGamma(Y(\dot{w}), \Lambda)[\ell(w)]$$
where $e_{\chi}$ is the projector for the action of ${\Tbf}^{w\Frob}$ on the block containing $\chi$. 
\end{thm}

\begin{rque}
It follows from this theorem that $\Ocal^{\DL}$ is a very rich source of uniform functions. There is however a lot more information contained in this category than in the space of uniform class functions. 
\end{rque}

\begin{defi}
A tilting representation of ${\Gbf}^{\Frob}$ is a complex of representations of the form 
$$\ch(T) \in \DD(\pt/{\Gbf}^{\Frob}, \Lambda)$$
where $T \in \Ocal^{\DL}$ is a perverse tilting object.
\end{defi}

The next theorem has also been discovered independently by Zhu \cite{ZhuLLC}.

\begin{thm}[Theorem \ref{thm:TiltingRepsAreProjective}]\label{thm:TiltingIntro}
\begin{enumerate}
\item The collection of tilting representations generates the category $\DD(\pt/{\Gbf}^{\Frob}, \Lambda)$.
\item Tilting representations are concentrated in degree $0$ and are projective representations.  
\end{enumerate}
\end{thm}

These objects provide a particularly nice collection of projective generators of the category $\DD(\pt/{\Gbf}^{\Frob}, \Lambda)$. It should be noted that we could have defined their characters previously in a formal way but the existence of $\Ocal^{\DL}$ allows us to work with actual representations instead of their characters. The construction of these tilting representations is a key example of the benefits of working with $\Ocal^{\DL}$ instead of uniform characters as they arise in a non-trivial way from the gluing of the various strata of $\Ocal^{\DL}$. 

\subsubsection{A conjecture of Dudas--Malle}

Classical category $\mathcal{O}$ is equipped with a `mixed refinement' $\Ocal^{\mix}$ and there is an involutive automorphism of $\Ocal^{\mix}$ 
$$\kappa : \Ocal^{\mix} \to \Ocal^{\mix}$$ 
called Koszul duality which exchanges simple and tilting objects, we refer to \cite{BGS}, \cite{TiltingExercises}, \cite{BezrukavnikovYun}. 

In the context of representations of finite groups of Lie type there is an involutive automorphism of the representation ring $R({\Gbf}^{\Frob})$ of ${\Gbf}^{\Frob}$ given by Alvis--Curtis duality, see \cite{AlvisDuality}, \cite{CurtisDuality}, \cite{DeligneLusztigDuality1} :
$$d : R({\Gbf}^{\Frob}) \to R({\Gbf}^{\Frob}).$$
It is known that in $R({\Gbf}^{\Frob})$ there is an equality of characters \cite{DeligneLusztigDuality2}. 
$$d(R_w^{\theta}) = (-1)^{\ell(w)}R_w^{\theta}.$$
More generally, Lusztig, in \cite{LusztigBook}, has given a formula for the Alvis--Curtis dual of $[\IH(\overline{X(w)}, \Qlb)]$. We reinterpret his calculation as follows 

\begin{thm}[Corollary \ref{cor:Duality}]
There is an equality 
$$ d([\IH(\overline{X(w)}, \Qlb)]) = \pm d(\ch (T_w)) $$
where $T_w$ is the indecomposable tilting sheaf containing $\Delta_{w,1}$ and supported on the closure of $\BwB$. 
\end{thm}
The calculations of Lusztig and our theorem admit refinements which take into account the weights of the Frobenius endomorphism acting on the cohomology (and intersection homology) of the Deligne--Lusztig varieties. This calculation suggests some deep interplay between Koszul duality for category $\Ocal$ and Alvis--Curtis duality. 

Finally, let $\delta > 0$ be the smallest integer such that $(\Gbf ,\Frob^{\delta})$ is split. Then $\Frob^{\delta}$ acts on the cohomology of $[\IH(\overline{X(w)}, \Qlb)]$. Let $\lambda \in \Qlb$ and denote by $[\IH(\overline{X(w)}, \Qlb)[\lambda]]$ the character of the generalized eigenspace of $\Frob^{\delta}$ on $\IH(\overline{X(w)}, \Qlb)$ for the eigenvalue $\lambda$.  Let $\overline{\lambda} \in \Flb$ and let $[\IH(\overline{X(w)}, \Qlb)[\overline{\lambda}]]$ be the character of 
$$[\IH(\overline{X(w)}, \Qlb)[\overline{\lambda}]] = \bigoplus_{\lambda} [\IH(\overline{X(w)}, \Qlb)[\lambda]]$$
where $\lambda$ ranges through the set of element of $\Zlb$ that reduce to $\overline{\lambda}$ in $\Flb$. 

\begin{thm}[Theorem \ref{thm:ConjDudasMalle}, \protect{\cite[Conjecture 1.2]{DudasMalle}}]\label{thm:DudasMalleIntro}
There exists $\ell_0$ depending only on $\Gbf$ (not on $\Frob$) such that for all $\ell \geq \ell_0$,  $\overline{\lambda} \in \Flb$, $w \in \Wbf$ the character of ${\Gbf}^{\Frob}$
$$d([\IH(\overline{X(w)}, \Qlb)[\overline{\lambda}]])$$
is the unipotent part of a projective $\Zlb$-character. 
\end{thm}

\begin{rque}
The hypothesis on $\ell$ in the theorem is to ensure that Kazhdan--Lusztig polynomials and $\ell$-Kazhdan--Lusztig polynomials agree. This ensure that if $T^{\Zlb}_w$ is a $\Zlb$ tilting representation, then the unipotent part of $T^{\Zlb}_w[\frac{1}{\ell}]$ is $T_w^{\Qlb}$ whose character is known. In general, the unipotent part of $T^{\Zlb}_w[\frac{1}{\ell}]$ decomposes as 
$$T^{\Qlb}_w \oplus_{v < w} T_{v}^{\Qlb, \oplus n_{w,v}}$$
where $n_{w,v} \in \mathbb{N}$ are certain multiplicities. 
\end{rque}

\subsubsection{Applications to decomposition numbers}

The conjecture of \cite{DudasMalle}, now theorem \ref{thm:ConjDudasMalle} was assumed to prove certain results concerning the unipotent decomposition numbers of ${\Gbf}^{\Frob}$. Now that this conjecture is established let us mention a few consequences. Let us fix $q$ and $\ell$ such that theorem \ref{thm:ConjDudasMalle} holds. 

\begin{enumerate}
\item Assume that condition $(ii)$ of \cite[Proposition 2.1]{DudasMalle} holds. Then the unipotent decomposition matrix of ${\Gbf}^{\Frob}$ has unitriangular shape and its entries do not depend on $q$ or $\ell$. Note that the unitriangularity has also been established by other methods in \cite{BrunatDudasTaylor}. 
\item The conjecture of \cite{DudasMalle} is assumed in calculations of \cite{DudasMalleharishChandra} and \cite{DudasMalleDecompMatrix} and is used to show that certain entries vanish. This now holds under our standing assumptions on $q$ and $\ell$. 
\end{enumerate}

\subsection{Notation and conventions}\label{sec:Notations}
We fix $\ell \neq p$ a prime and denote by $\Lambda \in \{\Qlb, \Zlb, \Flb\}$ a coefficient ring. For $X$ an algebraic stack, we denote by $\DD(X, \Lambda)$ the category of ind-constructible sheaves of $\Lambda$-modules as defined in \cite{HemoRicharzScholbach}. If $X$ is a finite type scheme then we have, 
$$\DD(X, \Lambda) = \Ind(\DD^b_c(X,\Lambda))$$
the ind-completion of the usual derived category of constructible sheaves on $X$. If the coefficient ring is clear, we will write $\DD(X, \Lambda) = \DD(X)$. Let $f : X \to Y$ be a morphism of stacks, whenever they are defined, we will denote by $f_*, f_!,f^*,f^!$ the usual pushforward and pullbacks. We fix a square root of $p$ in $\Zlb$. Whenever $X$ is a stack defined over $\Fq$, this defines a half Tate twist $(\frac{1}{2})$ on sheaves on $X_{\Fqb}$ with a Weil-structure. Given a stack $X$ with a Frobenius $\Frob : X \to X$, we denote by $\DD^b_c(X)^{\Weil}$ the category of Weil sheaves on $X$, i.e. the category of pairs $(A, \phi)$ where $A \in \DD^b_c(X)$ and $\phi : \Frob^*A \xrightarrow{\sim} A$.

 We denote by $\Hom(-,-)$ the derived functor of endomorphism and we denote by $\Hom^i$ its $i$-th cohomology groups and by $\End(X)$ the functor $\Hom(X,X)$ (where everything is derived). 

All categories are considered as $\infty$-categories but the $\infty$-part plays no essential role, all $\infty$-categories in this paper are either stable $\infty$-categories, as in \cite{LurieHA} or abelian/additive usual $1$-categories, the unfamiliar reader should ignore these technicalities and only consider triangulated categories. If $C$ is equipped with a $t$-structure, we denote by $C^{\heartsuit}$ its heart. We recall that an object $c \in C$ is compact if and only if 
$$\Hom_C(c,-)$$
commutes with arbitrary direct sums. In particular for a ring $A$, the compact objects of $\DD(A)$, the full derived category of $A$, are exactly the perfect complexes. They form a category which we denote by $\Perf(A) \subset \DD(A)$. A complex $K \in \DD(A)$ lies in $\Perf(A)$ if and only if it is quasi-isomorphic to a bounded complex of projectives of finite type.

\subsection{Acknowledgments}

The author thanks Olivier Dudas for suggesting the problem and for comments on a draft of this paper. We thank Jean-François Dat and Olivier Dudas for their support during the preparation of this paper. We thank Gunter Malle for doing a careful reading of a draft of this paper. This paper was written while the author was a guest at the Max Planck Institute for mathematics in Bonn. The author thanks Cédric Bonnafé and Gerhard Hiss for helpful discussions. The author thanks Xinwen Zhu for comments on a previous version of this paper. 

\section{Deligne--Lusztig category $\mathcal{O}$}\label{sec:DLCatO}

As in the introduction, we will consider the algebraic stack 
$$\frac{\Ubf \backslash \Gbf/\Ubf}{\Ad_{\Frob}\Tbf}$$ 
which we call the horocycle stack. 
We denote by $\Ocal^{\DL}_{\Lambda}$ the category $\DD(\frac{\Ubf \backslash \Gbf/\Ubf}{\Ad_{\Frob}\Tbf}, \Lambda)$. We call this category the Deligne--Lusztig category $\Ocal$. The goal of this section is to discuss the homological structure of this category. When the coefficients are clear, we will drop the index $\Lambda$ from the notation. We denote by $\Ocal_{\cons}^{\DL}$ the full subcategory of constructible sheaves. 

\subsection{Stratification}

We equip the horocycle stack with its Bruhat stratification, that is 
$$\frac{\Ubf \backslash \Gbf/\Ubf}{\Ad_{\Frob}\Tbf} = \bigsqcup_{w \in \Wbf} \frac{\Ubf \backslash \BwB /\Ubf}{\Ad_{\Frob}\Tbf}.$$
We denote by $j_w : \frac{\Ubf \backslash \BwB /\Ubf}{\Ad_{\Frob}\Tbf} \to \frac{\Ubf \backslash \Gbf/\Ubf}{\Ad_{\Frob}\Tbf}$ the inclusion of the stratum $w$, note that this is an affine immersion. 

\begin{lem}
Let $\dot{w} \in \Nbf(\Tbf)$ be a lift of $w$. There is an isomorphism of stacks 
$$\frac{\Ubf \backslash \BwB /\Ubf}{\Ad_{\Frob}\Tbf} = \pt/({\Tbf}^{w\Frob} \ltimes \Ubf_w)$$
where $\Ubf_w = \Ubf \cap \Ad(\dot{w})(\Ubf)$. 
\end{lem}

\begin{proof}
See \cite[Lemma 3.2.4]{EtevePaper1}.
\end{proof}

Since $\Ubf_w$ is a connected unipotent group, there is an equivalence of categories
$$\DD(\pt/({\Tbf}^{w\Frob} \rtimes \Ubf_w)) = \DD(\Rep_{\Lambda} {\Tbf}^{w\Frob})$$
see \cite[Section 3.2]{EtevePaper1} for a discussion on the normalization of this equivalence. 
This category decomposes into blocks 
$$\DD(\Rep_{\Lambda} {\Tbf}^{w\Frob}) = \oplus_{\chi} \DD^{\chi}(\Rep_{\Lambda} {\Tbf}^{w\Frob})$$
where $\chi$ ranges through the irreducible characters of ${\Tbf}^{w\Frob}$ over $\Qlb$ if $\Lambda = \Qlb$ and over $\Flb$ is $\Lambda = \Flb, \Zlb$. The unipotent block (with respect to $\Lambda$) is the block containing the trivial representation. This is usually also called the principal block.   

Let $\rho$ be a representation of ${\Tbf}^{w\Frob}$, we denote by 
\begin{enumerate}
\item $\Delta_w(\rho) = j_{w,!}\rho[\ell(w)](\frac{\ell(w)}{2})$,
\item $\nabla_w(\rho) = j_{w,*}\rho[\ell(w)](\frac{\ell(w)}{2})$.
\end{enumerate}
Since the $j_w$ are affine immersions, these sheaves are perverse sheaves in $\Ocal^{\DL}$. We have indicated the Tate twist in this definition in anticipation of section \ref{sec:MixedCategoryO}, they will not play a role before this section.

For a character $\chi : {\Tbf}^{w\Frob} \to \Lambda^{\times}$ of ${\Tbf}^{w\Frob}$, we denote by $E_{\chi}$ the projective cover of $\chi$ and by 
\begin{enumerate}
\item $\Delta_{w, \chi} = \Delta_{w}(E_{\chi})$,
\item $\nabla_{w,\chi} = \nabla_w(E_{\chi})$. 
\end{enumerate}
The objects $\Delta_{w,\chi}$ and $\nabla_{w, \chi}$ are called the standard and costandard objects respectively. Note that if $\Lambda = \Qlb$, then $\chi = E_{\chi}$. 

Finally, we will also denote for a representation $\rho$ of ${\Tbf}^{w\Frob}$
$$\IC_{w}(\rho) = j_{w,!*}\rho[\ell(w)](\frac{\ell(w)}{2}).$$

\begin{prop}
\begin{enumerate}
\item The category $\Ocal^{\DL, \heartsuit}$ is noetherian. 
\item If $\Lambda$ is a field, it is also Artinian and has finitely many irreducible objects which are the $\IC_{w}(\chi)$ where $\chi$ ranges through the set of irreducible representations of ${\Tbf}^{w\Frob}$. 
\end{enumerate}
\end{prop}

\begin{proof}
This is a direct application of \cite[Theorem 4.3.1]{BBD}.
\end{proof}

\begin{prop}\label{prop:compactGenCatO}
The category $\Ocal^{\DL}$ is compactly generated. The compact objects are exactly the sheaves that are constructible and compact when $*$-pullbacked (equivalently $!$-pullbacked) to all strata. In particular, the standard and costandard objects are compact. 
\end{prop}

We denote by $\Ocal^{\DL, \omega}$ the full subcategory of compact objects.

\begin{rque}
The reader should note that if $\Lambda \in \{\Zlb, \Flb\}$ then the inclusion $\Ocal^{\DL, \omega} \subset \Ocal^{\DL}_{\cons}$ can be strict. This already happens in the case when $G$ is a torus as the horocycle stack is then $\pt/{\Tbf}^{\Frob}$ and the category of sheaves on it is simply the category of representations of ${\Tbf}^{\Frob}$. The constructible sheaves on ${\Tbf}^{\Frob}$ are exactly the complexes of representations whose underlying complex of $\Lambda$-modules is perfect but if $\ell | |{\Tbf}^{\Frob}|$ then the trivial representation has infinite cohomological dimension and is therefore constructible but not compact. 
\end{rque}

\begin{rque}
It follows from Proposition \ref{prop:compactGenCatO} that if $\Lambda = \Qlb$ then the inclusion $\Ocal^{\DL, \omega} \subset \Ocal_{\cons}^{\DL}$ is an equivalence. 
\end{rque}

\begin{proof}[Proof of proposition \ref{prop:compactGenCatO}]
The category $\Ocal^{\DL}$ is generated by the objects $\Delta_{w, \chi}$. Since the immersions $j_w$ are of finite presentation both functors $j_w^!$ and $j_{w,*}$ are continuous hence their left adjoints preserve compact objects. If follows that the standard sheaves form a set of compact generators of the category. 

Let $A$ be a compact object. By the previous point $j_w^*A$ is compact and since in $\DD(\Rep_{\Lambda}{\Tbf}^{w\Frob})$ the compact objects are stable under Verdier duality, we deduce that $j_w^!A$ is compact. Conversely let $A$ be such that for all $w$, the objects $j_w^!A$ and $j_w^*A$ are compact. Let $w$ be such that $A$ is supported on $\BwB$ and this stratum is maximal for the inclusion and consider $C = \mathrm{cone}(j_{w,!}j_w^!A \to A)$ the cone of the adjunction map. Let $i_w$ be the inclusion of the union of the strata in the closure of $\BwB$, then $i_w$ is a closed immersion of finite presentation and therefore $i_w^*$ and $i_{w,*}$ both preserve compact objects, hence $C = i_{w,*}i_w^*A$ is compact. 

It remains to show that the costandard objects are compact. It is enough to show that $j_v^*\nabla_{w,\chi}$ is compact which follows from applying Verdier duality and the case of the standard sheaves. 
\end{proof}

\subsection{Link with Hecke categories}\label{sec:ComparisonHeckeCat}

Hecke categories attached to $\Gbf$ have two realizations : the equivariant Hecke categories and the free monodromic Hecke categories. These categories were studied, in particular, in \cite{BGS}, \cite{LusztigYun}, \cite{BezrukavnikovYun}, \cite{BezrukavnikovRicheSoergelTheory},  \cite{Gouttard}, \cite{EteveThesis}. Let us give a brief account of these categories. We denote by $\pi_1^t(\Tbf)$ the tame quotient of the étale fundamental group of the torus $\Tbf$ at the geometric point $1 \in \Tbf$. It is known that there is an isomorphism 
$$\pi_1^t(\Tbf) = X_*(\Tbf) \otimes \pi_1^t(\Gm) = X_*(\Tbf) \otimes \hat{\Z}^{(p)}(1).$$
Let $\chi, \chi' : \pi_1^t(\Tbf) \to \Lambda^{\times}$ be two characters of finite order and let us denote by $\mathcal{L}_{\chi}$ (reps. $\mathcal{L}_{\chi'}$) the corresponding Kummer sheaves on $\Tbf$. We can also pull them back to $\Bbf$ and understand them as sheaves on $\Bbf$. With this data, one can construct two categories : 
\begin{enumerate}
\item the category $\DD^b_c((\Bbf, \mathcal{L}_{\chi})\backslash \Gbf/(\Bbf, \mathcal{L}_{\chi'}))$ of $(\Bbf \times \Bbf, \mathcal{L}_{\chi} \otimes \mathcal{L}_{\chi'})$-equivariant sheaves on $\Gbf$, we refer to \cite{LusztigYun} for their construction and properties, 
\item the category $\DD^b_c((\Ubf, \mathcal{L}_{\chi}) \fatbslash \Gbf \fatslash (\Ubf, \mathcal{L}_{\chi'}))$ of free monodromic sheaves with generalized monodromy given by $(\chi, \chi')$, we refer to \cite{BezrukavnikovYun}, \cite{Gouttard}, \cite{EteveThesis} for a construction and their properties.
\end{enumerate}
In all the previous papers, the equivariant (resp. monodromic) Hecke category is the category 
$$\Hh^{\equi} = \oplus_{\chi,\chi'} \DD^b_c((\Bbf, \mathcal{L}_{\chi})\backslash \Gbf/(\Bbf, \mathcal{L}_{\chi'})), (\text{resp.} \Hh^{\mon} =  \oplus_{\chi,\chi'} \DD^b_c((\Ubf, \mathcal{L}_{\chi}) \fatbslash \Gbf \fatslash (\Ubf, \mathcal{L}_{\chi'}))).$$ 

In both settings, the category $\DD^b_c((\Bbf, \mathcal{L}_{\chi})\backslash \Gbf/(\Bbf, \mathcal{L}_{\chi'}))$ (reps. $\DD^b_c((\Ubf, \mathcal{L}_{\chi}) \fatbslash \Gbf \fatslash (\Ubf, \mathcal{L}_{\chi'})$) is nonzero only if $\chi$ and $\chi'$ are in the same $\Wbf$-orbit. Finally in both settings, there are standard and costandard objects $\Delta_{w, \chi}^{\equi}, \nabla_{w, \chi}^{\equi} \in \DD^b_c((\Bbf, \mathcal{L}_{\chi})\backslash \Gbf/(\Bbf, \mathcal{L}_{\chi'})$ (resp. $\Delta_{w, \chi}^{\mon}, \nabla_{w, \chi}^{\mon} \in \DD^b_c((\Ubf, \mathcal{L}_{\chi}) \fatbslash \Gbf \fatslash (\Ubf, \mathcal{L}_{\chi'})$). Note that we are not indicating $\chi'$ in the notation for these standard objects as we must have $\chi' = w\chi$. 

We now introduce functors between these Hecke categories and $\Ocal^{\DL}$. 
\begin{enumerate}
\item The free monodromic Hecke category is realized as a category of sheaves on $\Ubf \backslash \Gbf/\Ubf$, let $\mathfrak{p} :  \Ubf \backslash \Gbf/\Ubf \to \frac{\Ubf \backslash \Gbf/\Ubf}{\Ad_{\Frob}(\Tbf)}$ be the natural quotient map. We will consider $\mathfrak{p}_! : \Hh^{\mon} \to \Ocal^{\DL}$, 
\item Let $(\chi, \chi')$ be such that there exists an element $w \in \Wbf$ such that $\chi' = w\Frob(\chi)$, then there is a natural forgetful functor $\mathfrak{q}^* : \DD^b_c((\Bbf, \mathcal{L}_{\chi})\backslash \Gbf/(\Bbf, \mathcal{L}_{\chi'})) \to \Ocal^{\DL}$ which is obtained by forgetting the $(\Tbf \times \Tbf, \Lcal_{\chi} \otimes \Lcal_{\chi'})$-equivariance down to $\Tbf$-equivariance along the inclusion $\Tbf \to \Tbf \times \Tbf$ given by $t \mapsto (t, \Frob(t^{-1}))$. We denote it by $\mathfrak{q}^*$ as we think of this functor as a pullback along some map $\frac{\Ubf \backslash \Gbf/\Ubf}{\Ad_{\Frob}\Tbf} \to (\Bbf, \mathcal{L}_{\chi})\backslash \Gbf/(\Bbf, \mathcal{L}_{\chi'})$. 
\end{enumerate}

\begin{lem}\label{lem:LinkWithHeckeCat}
Let $w \in \Wbf$ and $\chi$ be a character of $\Tbf^{w\Frob}$. There are natural isomorphisms
\begin{enumerate}
\item $\mathfrak{p}_!\Delta_{w,\chi}^{\mon}[\dim \Tbf](\frac{\dim \Tbf}{2}) = \Delta_w(E_{\chi}), \mathfrak{p}_!\nabla_{w, \chi}^{\mon}[\dim \Tbf](\frac{\dim \Tbf}{2}) = \nabla_w(E_{\chi})$
\item $\mathfrak{q}^*\Delta_{w,\chi}^{\equi} = \Delta_w(\chi), \mathfrak{q}^*\nabla_{w,\chi}^{\equi} = \nabla_w(\chi)$.
\end{enumerate}
\end{lem}

\begin{proof}
The second point is clear since $\mathfrak{q}^*$ is simply the forgetful functor. The first point follows from \cite[Lemma 2.8.3]{EtevePaper1}. 
\end{proof}

\subsection{Tilting objects}

\begin{defi}
Let $A \in \Ocal^{\DL, \heartsuit}$, 
\begin{enumerate}
\item a $\Delta$-filtration on $A$ is a filtration whose graded pieces are standard sheaves, 
\item a $\nabla$-filtration on $A$ is a filtration whose graded pieces are costandard sheaves,
\item the sheaf $A$ is tilting if $A$ has both a $\Delta$ and a $\nabla$-filtration. 
\end{enumerate}
We denote by $\Ocal^{\DL}_{\tilt}$ the full subcategory of $\Ocal^{\DL, \heartsuit}$ of tilting objects. 
\end{defi}

\begin{thm}\label{thm:StructureCatO}
\begin{enumerate}
\item The category $\Ocal^{\DL, \heartsuit}$ has enough projective objects and all projectives have a $\Delta$-filtration. If $\Lambda$ is a field it also has enough injective objects and the injective objects have a $\nabla$-filtration.
\item The category $\Ocal^{\DL}_{\tilt}$ is an additive, idempotent complete and Krull-Schmidt category.
\item There is a unique bijection between indecomposable tilting objects and pairs $(w, \chi)$ where $w \in \Wbf$ and $\chi$ is a irreducible character (over the residue field of $\Lambda$) of $\Tbf^{w\Frob}$ characterized by the fact that the indecomposable tilting corresponding to $(w,\chi)$ is supported on the closure of $\BwB$ and has $\Delta_{w,\chi}$ with multiplicity one in a $\Delta$-filtration. We denote this tilting object by $T_{w, \chi}$. 
\end{enumerate}
\end{thm}

\begin{rque}
Theorem \ref{thm:StructureCatO} is very analogous to the structure theorem for category $\Ocal$, see \cite[Section 3]{BGS} and the proof uses the yoga of highest weight categories, see \cite[Section 7]{RicheHDR} and \cite{AcharRicheKoszulDuality1} for the $\Zl$-case. 
\end{rque}

\begin{proof}[Proof of Theorem \ref{thm:StructureCatO}]
The existence of enough projective for $\Ocal^{\DL, \heartsuit}$ follows from the same statement established for free monodromic Hecke categories and the next two observations.
\begin{enumerate}
\item If $P$ is a free monodromic projective object then $\mathfrak{p}_!P[\dim \Tbf]$ is perverse. Indeed free monodromic projective objects have a $\Delta^{\mon}$-filtration, see \cite{BezrukavnikovYun} and \cite{Gouttard}, hence $\mathfrak{p}_!P[\dim \Tbf]$ has a $\Delta$-filtration by lemma \ref{lem:LinkWithHeckeCat}.
\item The right adjoint functor $\mathfrak{p}^![-\dim \Tbf]$ is $t$-exact, as $\mathfrak{p}$ is smooth, hence $\Hom(\mathfrak{p}_!P, -)[-\dim \Tbf] = \Hom(P, \mathfrak{p}^!-)[-\dim \Tbf]$ is an exact functor.  
\end{enumerate}
The existence of a $\Delta$-filtration is therefore established for a generating family of projective objects. The existence for all projective objects follows from the characterization of objects with a $\Delta$-filtration of Lemma \ref{lem:CharDeltaFilt} and the fact that this characterization is stable under taking direct summands. The statement about injective objects follows by Verdier duality. This establishes $(i)$. 

For $(ii)$, Lemma \ref{lem:CharDeltaFilt} guarantees that the category $\Ocal^{\DL}_{\tilt}$ is idempotent complete and additive. The fact that it is Krull-Schmidt follows from the constructibility of tilting objects and the fact that $\Lambda$ is a local ring. 

For the characterization of indecomposable tilting objects, we first note that the functor $\mathfrak{p}_![\dim \Tbf]$ sends tilting objects to tilting objects by Lemma \ref{lem:LinkWithHeckeCat}. For free monodromic Hecke categories, the corresponding statement is established in \cite{BezrukavnikovYun} in the unipotent case and \cite{Gouttard} in the non-unipotent case. This guarantees the existence of a tilting sheaf supported on the closure of $\BwB$ and such that the multiplicity of $\Delta_{w, \chi}$ in one of its $\Delta$-flags is $1$. Since the category of tilting sheaves is Krull-Schmidt there exists a direct summand that is indecomposable and satisfies both of these properties. The uniqueness can be proved as in \cite[Theorem 7.14]{RicheHDR}.  
\end{proof}

\begin{rque}\label{rq:IndecomposibilityOfTilting}
Note that we prove the existence of a generating collection of tilting sheaves from the existence of tilting sheaves in the Hecke category. Also note that the pushforward of an indecomposable tilting sheaf is indecomposable. Indeed if $T \in \Hh^{\mon}$ is indecomposable then $\End(T)$ is a local algebra and $\End(\mathfrak{p}_!T)$ is a quotient of it hence it is also local.
\end{rque}

\begin{lem}\label{lem:CharDeltaFilt}
A sheaf $A \in \Ocal^{\DL}$ is tilting if and only if for all $w \in \Wbf$, the objects $j_w^!A$ and $j_w^*A$ are perverse and projective representations of $\Tbf^{w\Frob}$. More generally $A$ has a $\Delta$ (reps. $\nabla$)-filtration if and only if $j_w^*A$ is perverse and projective (resp. $j_w^!A$ is perverse and projective). 
\end{lem}

\begin{proof}
If $A$ has a $\Delta$-filtration, then $j_w^*A$ is filtered with graded pieces of the form $j_w^*\Delta_{v, \chi}$, where $v \in \Wbf$ and $\chi$ is a character of $\Tbf^{v\Frob}$. Since we have 
\begin{enumerate}
\item $j_w^*\Delta_{v, \chi} = 0$ if $v \neq w$, 
\item $j_w^*\Delta_{v, \chi} = E_{\chi}$ if $v = w$,
\end{enumerate}
it follows that $j_w^*$ is perverse and projective. 
Conversely if for all $w$, the sheaf $j_w^*A$ is perverse and projective, we now show that it is equipped with a $\Delta$-filtration. We equip $\Wbf$ with the Bruhat order, i.e. the order induced by the closure relations of strata and we choose an extension of this order to a total order, we denote this extension $\leq_{\tot}$. Let $w$ be the minimal element of $\Wbf$ such that $j_w^*A \neq 0$. We denote by $i_w : \frac{\Ubf \backslash \overline{\BwB}/\Ubf }{\Ad_{\Frob}\Tbf} \subset \frac{\Ubf \backslash \Gbf/\Ubf }{\Ad_{\Frob}\Tbf}$ the inclusion of the closure of the stratum corresponding to $w$ and by $k_w : V_w \subset  \frac{\Ubf \backslash \Gbf/\Ubf }{\Ad_{\Frob}\Tbf}$ the open complement. Since $w$ is minimal among the strata supporting $A$, we have 
$$j_{w,!}j_w^*A = i_{w,!}i_w^*A,$$
which is a direct sum of standard sheaves. 
We now have an excision triangle
$$k_{w,!}k_w^*A \to A \to i_{w,!}i_w^*A.$$
The map $A \to i_{w,!}i_w^*A$ is a map of perverse sheaves that is surjective in the category of perverse sheaves. Indeed the cokernel of this map is supported on $\frac{\Ubf \backslash {\BwB}/\Ubf }{\Ad_{\Frob}\Tbf}$ but the map $A \to i_{w,!}i_w^*A$ is an isomorphism when restricted to $\frac{\Ubf \backslash {\BwB}/\Ubf }{\Ad_{\Frob}\Tbf}$. It follows that this triangle is a short exact sequence of perverse sheaves. Finally, the sheaf $B = k_{w,!}k_w^*A$ satisfies
\begin{enumerate}
\item $j_v^*B = j_v^*A$ if $v >_{\tot} w$ and is perverse and projective by induction, 
\item $j_v^*B = 0$ if $v \leq_{\tot} w$. 
\end{enumerate}
Since $A$ is supported on finitely many strata (as $\Wbf$ is finite), by induction we can reduce to the case where $A$ is supported on a single stratum but then the claim is trivial. 
\end{proof}

\begin{lem}\label{lem:formalityHomTilting}
Let $T,T'$ be two tilting sheaves then $\Hom(T,T')$ is concentrated in degree $0$.
\end{lem}

\begin{proof}
By the five lemma, this reduces to the fact that $\Hom(\Delta, \nabla)$ is concentrated in degree $0$. 
\end{proof}

\begin{prop}\label{prop:generationTilting}
The category $\Ocal^{\DL, \omega}$ is generated by $\Ocal^{\DL}_{\tilt}$ as a triangulated category.
\end{prop}

\begin{proof}
This is clear since the standard and tilting objects generate the same category. 
\end{proof}

\begin{lem}
Let $T \in \Ocal_{\tilt, \Zlb}^{\DL}$ then $T \otimes_{\Zlb} \Qlb$ and $T \otimes_{\Zlb} \Flb$ are tilting. Moreover $T$ is indecomposable if and only if $T \otimes_{\Zlb} \Flb$ is indecomposable. 
\end{lem}

\begin{proof}
The object $T$ is indecomposable if and only if $\End(T)$ is a local algebra. As $T$ is tilting $\End(T)$ is a finite free $\Zlb$-module and the natural map $\End(T) \otimes_{\Zlb} \Flb \to \End(T \otimes_{\Zlb} \Flb)$ is an isomorphism. Hence $\End(T)$ is local if and only $\End(T) \otimes_{\Zlb} \Flb$ is local. 
\end{proof}

\subsection{Action of $\Tbf^{\vee}$} 

Consider the action of $\Tbf \times \Tbf$ on the horocycle stack by left and right translations. We denote by $\Tbf^{\vee}$ the torus dual to $\Tbf$ defined over $\Lambda$ and by $\Frob^{\vee} : \Tbf^{\vee} \to \Tbf^{\vee}$ the morphism dual to $\Frob$.
\begin{lem}
All sheaves on $\frac{\Ubf \backslash \Gbf/\Ubf}{\Ad_{\Frob} \Tbf}$ are $\Tbf \times \Tbf$-monodromic. Moreover the monodromy map yields a $(\Tbf^\vee\sslash \Wbf)^{\Frob^{\vee}}$-linear structure on $\Ocal^{\DL}$. 
\end{lem} 

\begin{proof}
Recall that monodromic sheaves are sheaves that are locally constant along $\Tbf \times \Tbf$-orbits and tame local systems, we refer to \cite[Section 2]{EtevePaper1} for a discussion. Since all strata are classifying spaces the monodromic part is clear. Monodromic sheaves are equipped with a canonical action of $\pi_1^t(\Tbf \times \Tbf)$, the tame fundamental group of $\Tbf \times \Tbf$. This groups is isomorphic to $X_*(\Tbf \times  \Tbf) \otimes \hat{\Z}^{(p)}(1)$. After fixing a trivialization of the roots of $1$ in $\Fqb$ we get a topological generator of $\hat{\Z}^{(p)}(1)$ and an action of $\Lambda[X_*(\Tbf \times \Tbf)] = \Ocal(\Tbf^{\vee}) \otimes_{\Lambda} \Ocal(\Tbf^{\vee})$ on monodromic sheaves. We first restrict this action to an action of $\Ocal(\Tbf^{\vee})^{\Wbf} \otimes_{\Lambda} \Ocal(\Tbf^{\vee})^{\Wbf}$. Finally we observe that the $\Ad_{\Frob}(\Tbf)$-equivariance shows that the map 
$$\Ocal(\Tbf^{\vee})^{\Wbf} \otimes_{\Lambda} \Ocal(\Tbf^{\vee})^{\Wbf} \to \maps(\id_{\Ocal^{\DL}})$$
factors through the quotient of $\Ocal({\Tbf}^{\vee})^{\Wbf} \otimes_{\Lambda} \Ocal({\Tbf}^{\vee})^{\Wbf}$ by the ideal $(f - \Frob^{\vee}(f), f \in \Ocal((\Tbf^{\vee})^{\Wbf})$, that is, by $\Ocal(\Tbf^{\vee} \sslash \Wbf)^{\Frob^{\vee}})$. 
\end{proof}

\begin{rque}
It should be noted that there is a bigger scheme acting on $\Ocal^{\DL}$. However, this finer structure will not be compatible when we go to $\Gbf^{\Frob}$-representations.
\end{rque}

\begin{lem}\label{lem:pointOverLambda}
There is a bijection, depending on a choice of a trivialization of the roots of $1$ in $\Fqb$, between $\Lambda$-points of $(\Tbf^{\vee}\sslash \Wbf)^{\Frob^{\vee}}$ and $\Frob^*$-stable semisimple $\Gbf^*$-conjugacy classes of elements of order invertible in $\Lambda$. We denote this set $\Gbf^*_{\Lambda}/\mathrm{ss}$. 
\end{lem}

\begin{proof}
Both sets are identified with the set of $\Wbf$-conjugacy classes of characters $X_*(\Tbf) \to \Lambda^{\times}$ that are stable under the Frobenius, see \cite[Corollary 5.24]{DeligneLusztig}. 
\end{proof}

\begin{corol}\label{corol:decompCatO}
The category $\Ocal^{\DL}$ splits as 
$$\Ocal^{\DL} = \bigoplus_{s \in (\Gbf^*_{\Lambda}/\mathrm{ss})^{\Frob^*}} \Ocal^{\DL,s}$$
where $\Ocal^{\DL,s}$ is the full subcategory of $\Ocal^{\DL}$ generated by $\Delta_{w, \chi}$ where the geometric conjugacy class of $(w, \chi)$ corresponds to $s$ (in the sense of \cite[Section 5]{DeligneLusztig}). 
\end{corol}

\begin{proof}
This follows immediately from Lemma \ref{lem:pointOverLambda} as the corresponding points index the connected components of $(\Tbf^{\vee}\sslash \Wbf)^{\Frob^{\vee}}$.
\end{proof}

\begin{rque}
In the following, we will denote by $\Ocal^{\DL, \unip}$ the direct summand corresponding to $s = 1$. 
\end{rque} 

\begin{lem}\label{lem:Decomp}
Let $A \in \Ocal^{\DL, \unip}_{\Zlb}$ then $A[\frac{1}{\ell}] = \oplus_{s_{\ell}} A_{s_{\ell}}$ where $s_{\ell}$ ranges through the set of conjugacy classes of element of order $\ell$. In particular we have
\begin{enumerate}
\item $\Delta_{w,1}^{\Zlb}[\frac{1}{\ell}] = \bigoplus_{\chi_{\ell}} \Delta_{w, \chi_{\ell}}^{\Qlb}$, 
\item $\nabla_{w,1}^{\Zlb}[\frac{1}{\ell}] = \bigoplus_{\chi_{\ell}} \nabla_{w, \chi_{\ell}}^{\Qlb}$, 
\item $T_{w, 1}^{\Zlb}[\frac{1}{\ell}] = \bigoplus_{\chi_{\ell}} T_{w, \chi_{\ell}}^{\Qlb} \oplus \bigoplus_{v < w, \chi_{\ell}} T_{v, \chi_{\ell}}^{n_{v,w,\chi_{\ell}}}$, where $n_{v,w,\chi_{\ell}}$ is some multiplicity. 
\end{enumerate}
In the previous sums the character $\chi_{\ell}$ ranges over characters of $\ell^{\infty}$-torsion of $\Tbf^{w\Frob}$. 
\end{lem}

\begin{proof}
The first point follows from the comparison between the $\Zlb$ and $\Qlb$-versions of Lemma \ref{lem:pointOverLambda} and Corollary \ref{corol:decompCatO}. The decompositions of $\Delta_{w,1}^{\Zlb}$ and $\nabla_{w,1}^{\Zlb}$ follows from the compatibility of the formation of the functors $j_{w,!}$ and $j_{w,*}$ with change of coefficients. The decomposition of the tilting objects follows from the first observation that $T_{w,1}^{\Zlb}[\frac{1}{\ell}]$ is tilting by the decomposition of the $\Delta$ and $\nabla$. Furthermore $T_{w,1}^{\Zlb}[\frac{1}{\ell}]$ is supported on the closure of $\BwB$ and contains each of the $\Delta_{w,\chi_{\ell}}^{\Qlb}$ with multiplicity one, the rest of the statement follows from the classification Theorem \ref{thm:StructureCatO}.
\end{proof}

\subsection{Mixed Deligne--Lusztig category $\Ocal$}\label{sec:MixedCategoryO}

We equip the stack $\frac{\Ubf \backslash \Gbf/\Ubf}{\Ad_{\Frob} \Tbf}$ with the endomorphism $\Frob^{\delta}$ where $\delta$ is minimal such that the pair $(\Gbf, \Frob^{\delta})$ is a split reductive group over $\F_{q^{\delta}}$. All mixed structures are done with respect to this Frobenius. Furthermore, the Bruhat stratification is stable under $\Frob^{\delta}$. We define the category $\Ocal^{\DL, \mix}$ to be the full subcategory of constructible mixed $\Qlb$-sheaves on $\frac{\Ubf \backslash \Gbf/\Ubf}{\Ad_{\Frob}\Tbf}$ generated as a triangulated category by Tate twists of pure $\IC$-sheaves of weight $0$. There is a natural forgetful functor 
$$\omega : \Ocal^{\DL, \mix} \to \Ocal_{\cons}^{\DL}.$$

\begin{lem}\label{lem:ExistenceMixedRefinements}
For all pairs $(w, \chi)$ where $w \in \Wbf$ and $\chi$ is a character of $\Tbf^{w\Frob}$, there exist objects $\Delta_{w,\chi}^{\mix}, \nabla_{w,\chi}^{\mix}, \IC_{w, \chi}^{\mix}$ and $T_{w,\chi}^{\mix}$ whose images under $\omega$ are respectively $\Delta_{w,\chi}, \nabla_{w,\chi}, \IC_{w, \chi}$ and $T_{w,\chi}$. 
\end{lem}

\begin{proof}
The statement for the standard and costandard sheaves follows from the fact that the functors $j_{w,!}$ and $j_{w,*}$ preserve mixed sheaves. The statement for the $\IC$-sheaves is clear by \cite{BBD}. Finally the statement for the tilting follows from \cite{BezrukavnikovYun}, as the tilting sheaves are obtained by pushing along $\Ubf \backslash \Gbf/\Ubf \to \frac{\Ubf \backslash \Gbf/\Ubf}{\Ad_{\Frob}\Tbf}$ the free monodromic tilting sheaves which are mixed. 
\end{proof}

\begin{rque}
The mixed refinements in Lemma \ref{lem:ExistenceMixedRefinements} are uniquely characterized up to isomorphism by the fact that all four collections of objects are perverse, the sheaf $\IC_{w, \chi}^{\mix}$ is a simple mixed perverse sheaf and the object $T_{w,\chi}^{\mix}$ is a mixed tilting sheaf. 
\end{rque}

Let $\Hh$ denote the generic Hecke algebra of $\Wbf$ over $\Z[v^{\pm 1}]$. This is the algebra generated by elements $(T_w)$ subject to the relations 
\begin{enumerate}
\item $T_wT_{w'} = T_{ww'}$ if $\ell(ww') = \ell(w) + \ell(w')$, 
\item $T_s^2 = v^{-2}T_e + (v^{-2}-1)T_s$ if $s \in \Wbf$ is a simple reflection. 
\end{enumerate}
We denote by $H_w = v^{\ell(w)}T_w$ and we use the notation of \cite{SoergelHeckeAlgebra}. Recall from \cite{KazhdanLusztig} that there is an anti-involution 
$$\overline{(-)} : \Hh \to \Hh$$ 
satisfying $\overline{v} = v^{-1}$ and $\overline{H_w}= H^{-1}_{w^{-1}}$. There are now two bases of $\Hh$ that are self dual, they are denoted by $(C_w)$ and $(C'_w)$ in \emph{loc. cit.}. We shall use the notation of \cite{SoergelHeckeAlgebra}. 
\begin{thm}[\cite{KazhdanLusztig}]
For all $w \in \Wbf$, there exist two unique self dual elements $\underline{H}_w$ and $\underline{\tilde{H}}_w$ of $\Hh$ such that 
\begin{enumerate}
\item $\underline{H}_w \in H_w + \sum_{y < w} v\Z[v]H_y$, 
\item $\underline{\tilde{H}}_w \in H_w + \sum_{y < w} v^{-1}\Z[v^{-1}]H_y$.
\end{enumerate}
\end{thm}

\begin{rque}[\protect{\cite[Remark 2.4]{SoergelHeckeAlgebra}}]\label{rq:Defb}
We denote by $b : \Hh \to \Hh$ the involution defined by 
$$b(H_w) = H_w, b(v) = -v^{-1}.$$ Then there is an equality
$$b(\underline{H}_w) = \underline{\tilde{H}}_w.$$
\end{rque}

\begin{lem}\label{lem:Categorification}
There is an isomorphism of abelian groups 
$$[-] : \Hh \simeq K_0(\Ocal^{\DL, \mix, \unip}_{\Qlb}).$$
Moreover the following holds 
$$[\Delta_{w,1}^{\mix}] = H_w, [\nabla_{w,1}^{\mix}] = H_{w^{-1}}^{-1}, [T_{w,1}^{\mix}] = \underline{H}_w, [\IC_{w,1}^{\mix}] = \underline{\tilde{H}}_w, [-(\frac{1}{2})] = v^{-1}[-].$$
\end{lem}

\begin{proof}
Since $\Ocal^{\DL, \mix, \unip}$ is stratified by the categories $\DD(\pt/\Tbf^{w\Frob})^{\unip, \mix} \simeq \DD(\pt)^{\mix}$, we get an isomorphism 
$$\Hh \simeq K_0(\Ocal^{\DL, \mix, \unip}_{\Qlb})$$
satisfying $[\Delta_{w,1}^{\mix}] = H_w$. 
Similarly, there are isomorphisms between $\Hh$ and the  Grothendieck groups of the two mixed realizations of the Hecke categories $\Hh^{\eq, \unip, \mix}$ and $\Hh^{\mon, \unip, \mix}$ constructed in \cite{BezrukavnikovYun}, that is we have 
$$K_0(\Hh^{\eq, \unip, \mix}) = \Hh = K_0(\Hh^{\mon, \unip, \mix}).$$
The quotient maps 
$$\Ubf \backslash \Gbf/\Ubf \to \frac{\Ubf \backslash \Gbf/\Ubf}{\Ad_{\Frob}\Tbf} \to \Bbf \backslash \Gbf/\Bbf$$
yield isomorphisms of Grothendieck groups (induced by pushforward for the first one and pullback for the second one) 
$$K_0(\Hh^{\mon, \unip, \mix}) = K_0(\Ocal^{\DL,\unip, \mix}) = K_0(\Hh^{\eq, \unip, \mix}).$$
Since the pushforward $\Hh^{\mon, \unip, \mix} \to \Ocal^{\DL, \unip, \mix}$ sends standard (resp. costandard, resp. tilting) to standard (resp. costandard, resp. tilting), we deduce from the categorification theorem of \emph{loc. cit.} that  $[\nabla_{w,1}^{\mix}] = H_{w^{-1}}^{-1}$ and $[T_{w,1}^{\mix}] = \underline{H}_w$. Similarly, the pullback $\Hh^{\eq, \unip, \mix} \to \Ocal^{\DL, \unip, \mix}$ sends standard (resp. costandard, resp. simple) to standard (resp. costandard, resp. simple) which allows us to deduce that $[\IC_{w,1}^{\mix}] = \underline{\tilde{H}}_w$.
\end{proof}

\begin{rque}
The reader should note that while there is a ring structure on $\Hh_q$, there is a priori no ring structure or action defined on the RHS. However this equivalence is linear over $\Z[v^{\pm 1}]$.
\end{rque}

More generally let $\Hh_{\mon}$ be the monodromic Hecke algebra defined in \cite[1.4]{LusztigConjClasses} and see \cite[Section 3.14]{LusztigYun} for an account. This is the $\Z[v^{\pm 1}]$-algebra with generators $(T_w, 1_{\chi})$ where $\chi \in \Ch(T)$ subject to the following relations 
\begin{enumerate}
\item $T_{w}T_{w'} = T_{ww'}$ if $\ell(w) + \ell(w') = \ell(ww')$, 
\item $1_{\chi}1_{\chi'} = \delta_{\chi,\chi'}1_{\chi}$,
\item $T_w1_{\chi} = 1_{w\chi}T_w$, 
\item $T_s^2 = v^2T_s + (v^2 - 1)\sum_{s \in \Wbf_s^{\circ}}T_s1_{\chi}$, we refer to \cite{LusztigYun} for the notation $\Wbf_s^{\circ}$. 
\item $\sum_{\chi} 1_{\chi} = 1$, note that combined with the second relation the elements $1_{\chi}$ form a set of orthogonal idempotents and the sum makes sense. 
\end{enumerate}
We set $H_w = v^{-\ell(w)}T_w$ and we note that $\Hh$ is the direct summand attached to the idempotent $1_{\triv}$. We will also denote by $H_{w, \chi} = H_w1_{\chi}$. 

As in the unipotent case, this Hecke algebra is equipped with an involution 
$$\overline{(-)} : \Hh_{\mon} \to \Hh_{\mon}, v^nH_{w}1_{\chi} \mapsto v^{-n}H^{-1}_{w^{-1}}1_{\chi}.$$
We refer to \cite[Section 3.14]{LusztigYun} for a discussion on the proof of the next theorem. 
\begin{thm}
There exist two self-dual bases $\underline{H}_{w, \chi}$ and $\underline{\tilde{H}}_{w,\chi}$ of $\Hh_{\mon}$ such that 
\begin{enumerate}
\item $\underline{H}_{w, \chi} \in H_{w, \chi} + \sum_{y < w, \chi'} v\Z[v]H_{y, \chi'}$, 
\item $\underline{\tilde{H}}_{w,\chi} \in H_{w, \chi} + \sum_{y < w, \chi'} v^{-1}\Z[v^{-1}]H_{y, \chi'}$.
\end{enumerate}
\end{thm}

As in the monodromic case we have a categorification theorem 
\begin{thm}\label{thm:ThmCategorificationMonodromicHeckeAlg}
There exists an injective map 
$$K_0(\Ocal^{\DL, \mix}) \to \Hh_{\mon}$$
such that 
$$[\Delta_{w,\chi}^{\mix}] = H_{w,\chi}, [\nabla_{w,\chi}^{\mix}] = H_{w^{-1},\chi}^{-1}, [\IC_{w,\chi}^{\mix}] = \underline{\tilde{H}}_{w,\chi}.$$
\end{thm}

\begin{proof}
As in the unipotent case, these statements can be deduced from the corresponding statement from the equivariant categories of \cite{LusztigYun}, see in particular Section 3.14 of \emph{loc.cit.} 
\end{proof}

\begin{conjecture}\label{conj:CharacterMixedTilting}
Using the same notation as in Theorem \ref{thm:ThmCategorificationMonodromicHeckeAlg} we have 
$$[T_{w,\chi}^{\mix}] = \underline{H}_{w,\chi}.$$
\end{conjecture}

The proof of this conjecture in the unipotent case relied on the study of free monodromic mixed unipotent sheaves as done in \cite{BezrukavnikovYun}. We did not find a source in the literature that does this in the non-unipotent mixed case. In the thesis of Gouttard \cite{Gouttard}, the author studies the non-unipotent case but not in the mixed setting, it can be deduced from the categorification theorems of \emph{loc. cit.} that 
$$[T_{w,\chi}](1) = \underline{H}_{w,\chi}(1).$$
We expect that the generalization of \cite{BezrukavnikovYun} to the non-unipotent setting should be a matter of adapting the arguments in the correct framework. 

\subsection{Relation with category $\mathcal{O}$}\label{sec:ComparisonCatO}

The goal of this section is to relate the decomposition numbers appearing in Lemma \ref{lem:Decomp} to decomposition numbers in the category $\mathcal{O}$. Recall that category $\Ocal$ is the category of sheaves on the stack 
$$\Ubf \backslash \Gbf/\Bbf.$$
For a coefficient ring $\Lambda$, we will denote by 
$$\Ocal_{\Lambda} = \DD(\Ubf \backslash \Gbf/\Bbf, \Lambda),$$
as in the case of $\Ocal^{\DL}$, we will drop the index $\Lambda$ when the coefficients are clear. We also equip this category with its Bruhat stratification $\Ubf \backslash \Gbf/\Bbf = \sqcup_w \Ubf \backslash \BwB/\Bbf$ and we denote by $j_w : \Ubf \backslash \BwB/\Bbf \subset \Ubf \backslash \Gbf/\Bbf$ the inclusion of the stratum corresponding to $w$. 
We recall from \cite{AcharRicheKoszulDuality1}, \cite{RicheSoergelWilliamson} that we have the following objects
\begin{enumerate}
\item the standard sheaves $\Delta_w^{\Ocal, \Lambda} = j_{w,!}\Lambda[\ell(w)]$, where $\Lambda$ denotes the constant sheaf, 
\item the costandard sheaves $\nabla_w^{\Ocal, \Lambda} = j_{w,*}\Lambda[\ell(w)]$, 
\item the intersection complexes $\IC_w^{\Ocal, \Lambda} = j_{w,!*}\Lambda[\ell(w)]$, 
\item the indecomposable tilting sheaves $T_w^{\Ocal, \Lambda}$, which exist by \cite[Appendix B]{AcharRicheKoszulDuality1}. 
\end{enumerate}
As explained in \cite[Section 2.7]{AcharRicheKoszulDuality1}, if $T \in \Ocal^{\heartsuit}$ is a tilting sheaf, then $T[\frac{1}{\ell}]$ is also a tilting sheaf. Consequently, in \emph{loc. cit.} they introduce multiplicities 
$$T_w^{\Ocal, \Zlb}[\frac{1}{\ell}] = T_w^{\Ocal, \Qlb} \oplus \bigoplus_{v < w} (T_v^{\Ocal, \Qlb})^{\oplus n^{\Ocal}_{v,w}}.$$

\begin{lem}\label{lem:EqualityDecompNumbers}
For all $v,w \in W$, there is an equality of multiplicities
$$n_{v,w}^{\Ocal} = n_{v,w,1}$$
where the integer in the RHS is the multiplicity defined in Lemma \ref{lem:Decomp}. 
\end{lem}

We consider the categories of unipotent free monodromic sheaves on $\Ubf \backslash \Gbf/\Ubf$ as constructed in \cite{BezrukavnikovYun} and \cite{Gouttard} for $\Lambda = \Qlb$ and $\Lambda = \Zlb$, we denote these categories by $\Hh^{\mon,\unip}_{\Zlb}$ and $\Hh^{\mon, \unip}_{\Qlb}$. Moreover, there is an inversion of $\ell$-functor 
$$\inv_{\ell} : \Hh^{\mon,\unip}_{\Zlb} \to \Hh^{\mon, \unip}_{\Qlb}$$
which is constructed in \cite[Section 2.2]{EtevePaper1} essentially as composition of $\Qlb \otimes_{\Zlb}-$ and a certain completion. 
We denote by 
$$\mathfrak{p}' : \Ubf \backslash \Gbf/\Ubf \to \Ubf \backslash \Gbf/\Bbf$$
the quotient map.

\begin{proof}
Consider the following diagram of categories 
\[\begin{tikzcd}
	& {\Hh^{\mon, \unip, \Zlb}} \\
	{\Ocal^{\DL, \Zlb}} & {\Hh^{\mon, \unip, \Qlb}} & {\Ocal^{\Zlb}} \\
	{\Ocal^{\DL, \Qlb}} && {\Ocal^{\Qlb}.}
	\arrow["{\mathfrak{p}_![\dim \Tbf]}"{description}, from=1-2, to=2-1]
	\arrow["{\inv_{\ell}}"{description}, from=1-2, to=2-2]
	\arrow["{\mathfrak{p}'_![\dim \Tbf]}"{description}, from=1-2, to=2-3]
	\arrow["{\Qlb \otimes_{\Zlb}-}"', from=2-1, to=3-1]
	\arrow["{\mathfrak{p}_![\dim \Tbf]}"{description}, from=2-2, to=3-1]
	\arrow["{\mathfrak{p}'_![\dim \Tbf]}"{description}, from=2-2, to=3-3]
	\arrow["{\Qlb \otimes_{\Zlb}-}", from=2-3, to=3-3]
\end{tikzcd}\]
Let $T \in \Hh^{\mon, \unip, \Zlb}$ be a free monodromic tilting object, then the object $\inv_{\ell}T \in \Hh^{\mon, \unip, \Qlb}$ also free monodromic tilting. Furthermore the following relations hold by definition of the functor $\inv_{\ell}$. We have 
$$\mathfrak{p}_!\inv_{\ell}T[\dim \Tbf] = ((\mathfrak{p}_!T)[\frac{1}{\ell}])^{s = 1}[\dim \Tbf]$$ 
where $^{\unip}$ denotes the direct summand of $(p_!T)[\frac{1}{\ell}])$ corresponding to $s = 1$ in the decomposition of lemma \ref{lem:Decomp}. We also have
$$\mathfrak{p}'_!\inv_{\ell}T[\dim \Tbf] = (\mathfrak{p}'_!T)[\frac{1}{\ell}][\dim \Tbf].$$
Let $\hat{T}_w^{\Zlb}$ be the indcomposable free monodromic tilting sheaf in $\Hh^{\mon, \unip, \Zlb}$ corresponding to $w$. By Remark \ref{rq:IndecomposibilityOfTilting}, we have 
$$\mathfrak{p}_!\hat{T}_w^{\Zlb}[\dim \Tbf] = T_{w,1}^{\Zlb}$$ 
and by \cite[Lemma 5.8]{BezrukavnikovRicheSoergelTheory}, we have 
$$\mathfrak{p}'_!\hat{T}_w^{\Zlb}[\dim \Tbf] = T_w^{\Ocal, \Zlb}.$$
Now since the category of free monodromic tilting sheaves in $ \Hh^{\mon, \unip, \Qlb}$ is Krull-Schmidt there is a decomposition
$$\inv_{\ell}\hat{T}_w^{\Zlb} = \hat{T}^{\Qlb}_{w} \oplus \bigoplus_{v<w} (\hat{T}^{\Qlb}_{v})^{\oplus \hat{n}_{v,w}}.$$
It now follows that $n_{v,w,1} = \hat{n}_{v,w} = n^{\Ocal}_{v,w}$. 
\end{proof}

\begin{rque}
The decomposition numbers that appear in Lemma \ref{lem:EqualityDecompNumbers} control the difference between the usual Kazhdan--Lusztig basis in $\Hh$ and the $\ell$-Kazhdan--Lusztig basis (which depends on $\Gbf$ and not just on $\Wbf$), firstly defined in the case of crystallographic Coxeter groups in \cite{ParitySheaves}. We refer to \cite{pcanonicalBasis} for an account on this basis. 
\end{rque}

\begin{lem}\label{lem:IncomposabilityLargeL}
For a fixed $\Gbf$, there exists $\ell_0$ large enough such that for all $\ell \geq \ell_0$ the Kazhdan-Lusztig and $\ell$-Kazhdan--Lusztig basis coincide. For such $\ell$, we have $n_{v,w,1} = 0$ for all $v < w$.
\end{lem}

\begin{proof}
By \cite[Theorem 2.6]{AcharRicheKoszulDuality1}, if $\ell$ is good for $\Gbf$, these decomposition numbers are equal to decomposition numbers of parity sheaves. By \cite[Proposition 2.41]{ParitySheaves}, mod $\ell$ indecomposable parity sheaves and $\IC$ sheaves coincide for all but finitely many primes $\ell$.
\end{proof}

\begin{rque}
Predicting the behaviour of this basis and in particular when it differs from the usual Kazhdan--Lusztig basis is a very hard problem. In \cite{CalculationPCanonicalBasis}, the authors construct an algorithm to compute this basis. 
\end{rque}

\section{Tilting representations}\label{sec:TiltingRep}

\subsection{Horocycle correspondence}

We introduce the horocycle correspondence which we attribute to Lusztig. This same correspondence (not presented in a stacky and twisted way) was used to first define character sheaves \cite{CharacterSheaves1}. 

$$\frac{\Gbf}{\Ad_{\Frob}\Gbf} \xleftarrow{\alpha} \frac{\Gbf}{\Ad_{\Frob}\Bbf} \xrightarrow{\beta} \frac{\Ubf \backslash \Gbf/\Ubf}{\Ad_{\Frob}\Tbf}.$$

We introduce the following functors (which stand for horocycle and character) 
$$\ch = \alpha_!\beta^*[\dim \Ubf](\frac{\dim \Ubf}{2}) : \Ocal^{\DL} \to \DD(\Rep_{\Lambda} \Gbf^{\Frob})$$
and 
$$\hc = \alpha_!\beta^*[-\dim \Ubf](-\frac{\dim \Ubf}{2}) : \DD(\Rep_{\Lambda} \Gbf^{\Frob}) \to \Ocal^{\DL}.$$
Note that since $\alpha$ is proper and smooth, and $\beta$ is smooth the functor $\ch$ is right adjoint to $\hc$.

\begin{thm}[\protect{\cite[Corollary 3.3.5, Lemma 3.2.7]{EtevePaper1}}]\label{thm:ConservativityAndDeligneLusztig} 
\begin{enumerate}
\item The functor $\hc$ is conservative.
\item There is an isomorphism of functors $ \ch j_{w,!} = R_w[\ell(w)]$ where the RHS is the Deligne--Lusztig induction functor. 
\end{enumerate}
\end{thm}

Let us also record the compatibility with (geometric) Lusztig series.

\begin{thm}
\begin{enumerate}
\item The category $\DD(\Rep_{\Lambda} \Gbf^{\Frob})$ splits as 
$$\DD(\Rep_{\Lambda} \Gbf^{\Frob}) = \bigoplus_{s \in (G_{\Lambda}^*/\mathrm{ss})^{\Frob^*}} \DD(\Rep^s_{\Lambda} \Gbf^{\Frob})$$
where $\DD(\Rep^s_{\Lambda} \Gbf^{\Frob})$ is the Lusztig series corresponding to the conjugacy class $s$. 
\item The functors $\ch$ and $\hc$ preserve the splittings into Lusztig series and the splitting of $\Ocal^{\DL}$ of Corollary \ref{corol:decompCatO}.  
\end{enumerate}
\end{thm}

\begin{proof}
The first point is done in the original paper of Deligne and Lusztig \cite{DeligneLusztig} if $\Lambda = \Qlb$. The integral and the modular cases are done in \cite[Theorem 2.2]{BroueMichel}. In view of Theorem \ref{thm:ConservativityAndDeligneLusztig}, the second point follows from the definition of the Lusztig series and the compatibility between the Deligne--Lusztig functors and the functors $\ch$ and $\hc$ of theorem \ref{thm:ConservativityAndDeligneLusztig}. 
\end{proof}
\subsection{Tilting representations}

\begin{defi}
A tilting representation of $\Gbf^{\Frob}$ is an object in $\DD(\pt/\Gbf^{\Frob}, \Lambda)$ of the form $\ch(T)$ where $T \in \Ocal_{\tilt}^{\DL}$. 
\end{defi}

\begin{lem}\label{lem:Affineness}
All the complexes $\ch(\bigoplus_{\chi}\Delta_{w,\chi})$ are concentrated in nonnegative degree. 
\end{lem}

\begin{proof}
This is a consequence of Artin vanishing as all Deligne--Lusztig varieties are affine.
If $q$ is large enough then this is \cite[Theorem 9.7]{DeligneLusztig}, the general case is done in \cite[Corollary 3.12]{AffinenessDL}.
\end{proof}

\begin{rque}
Another proof of Lemma \ref{lem:Affineness} is done in \cite{ZhuLLC} and is independent of \cite{AffinenessDL}.
\end{rque} 

\begin{thm}\label{thm:TiltingRepsAreProjective}\footnote{This result was obtained independently by Zhu \cite{ZhuLLC}}
Let $T \in \Ocal_{\tilt}^{\DL}$.
\begin{enumerate}
\item The complex $\ch(T)$ is a compact object of $\DD(\Rep_{\Lambda} \Gbf^{\Frob})$.
\item The complex $\ch(T)$ is concentrated in degree $0$ and is a projective object of $\Rep_{\Lambda}\Gbf^{\Frob}$. 
\end{enumerate}
\end{thm}

\begin{proof}
Since the map $\alpha$ is proper and smooth and the map $\beta$ is smooth the functor $\ch$ has a continuous right adjoint hence $\ch$ preserves compact objects. Since tilting sheaves are compact by Proposition \ref{prop:compactGenCatO}, the objects $\ch(T)$ are compact.

By Theorem \ref{thm:ConservativityAndDeligneLusztig} and Lemma \ref{lem:Affineness}, it follows that $\ch(\Delta) \in \DD^{\geq 0}(\pt/\Gbf^{\Frob})$ (resp. $\ch(\nabla) \in \DD^{\leq 0}(\pt/\Gbf^{\Frob})$). Since $T$ has both a $\Delta$ and a $\nabla$-filtration, we get that $\ch(T) \in \DD^{\heartsuit}(\pt/\Gbf^{\Frob})$. Hence $\ch(T)$ is a representation of $\Gbf^{\Frob}$ concentrated in degree $0$, which is compact in $\DD(\pt/\Gbf^{\Frob})$. Since $\Lambda[\Gbf^{\Frob}]$ is a symmetric algebra, it follows that the representation $\ch(T)$ is projective. 
\end{proof}

\begin{prop}
The representations $\ch(T)$ for $T \in \Ocal_{\tilt}^{\DL}$ generate the triangulated category $\Perf(\Lambda[\Gbf^{\Frob}])$.
\end{prop}

\begin{proof}
Since the $\ch(T)$ generate the category $\Ocal_{\tilt}^{\DL}$ this statement is equivalent to the conservativity property of the functor $\hc$ of Theorem \ref{thm:ConservativityAndDeligneLusztig}.  
\end{proof}

\section{Application to Alvis--Curtis duality}

\subsection{The Alvis--Curtis dual of an $\mathrm{IC}$ representation}

We consider the category $\DDc(\pt/\Gbf^{\Frob}, \Qlb)^{\mix}$ of representations of $\Gbf^{\Frob}$ such that the underlying complex of vector spaces is mixed, where we use, as in Section \ref{sec:MixedCategoryO}, the Frobenius $\Frob^{\delta}$.  

We consider the Alvis--Curtis duality functor, see \cite{DeligneLusztigDuality1}
$$d : \DDc(\pt/\Gbf^{\Frob}, \Qlb)^{\mix} \to \DDc(\pt/\Gbf^{\Frob}, \Qlb)^{\mix}.$$
We denote by $R(\Gbf^{\Frob})$ the Grothendieck group of $\DDc(\Gbf^{\Frob},\Qlb)$ and we consider the ring $\Z[v^{\pm 1}] \otimes R(\Gbf^{\Frob})$. There is natural map 
$$\gamma : K_0(\DDc(\pt/\Gbf^{\Frob}, \Qlb)^{\mix}) \to \Z[v^{\pm 1}] \otimes R(\Gbf^{\Frob})$$
which sends the class of a representation $\rho = \oplus_j \rho_j$ where $\rho_j$ is the generalized eigenspace of Frobenius of weight $j$ to $\sum_j v^j \otimes [\rho_j]$. The group $K_0(\DDc(\pt/\Gbf^{\Frob}, \Qlb)^{\mix})$ is $\Z[v^{\pm 1}]$-linear where $v^{-1}$ acts by the half-Tate twist and the map $\gamma$ is $\Z[v^{\pm 1}]$-linear. 

\begin{rque}
The functor $\gamma \circ \ch$ induces a $\Z[v^{\pm 1}]$-linear morphism 
$$K_0(\Ocal^{\DL, \mix}_{\Qlb}) \to \Z[v^{\pm 1}] \otimes R(\Gbf^{\Frob}).$$
\end{rque}

Let $a : \Hh \to \Hh$ be the automorphism of the Hecke algebra given by 
$$a(v) = v, a(H_x) = (-1)^{\ell(x)}H_{x^{-1}}^{-1}.$$
Note that we have $\underline{H}_w = (-1)^{\ell(w)}a(\underline{\tilde{H}}_w)$, see \cite[Theorem 2.7]{SoergelHeckeAlgebra}. 

We will denote, as in \cite[Section 3.1]{LusztigBook}, $\Irr(\Wbf)_{\ex}$ the set of irreducible modules of $\Wbf$ which can be extended to a $\Wbf \rtimes \langle \Frob \rangle$-module, and for $\widetilde{E} \in \Irr(\Wbf)_{\ex}$ we denote by $\widetilde{E}(v)$ the corresponding $\Q(v)$-irreducible $\Hh \otimes_{\Z[v^{\pm 1}]} \Q(v)$-representation. Finally given $\widetilde{E} \in \Irr(\Wbf)_{\ex}$, we denote by $R_{\widetilde{E}}$ the corresponding almost character of $\Gbf^{\Frob}$ constructed in \cite{LusztigBook} and we consider the following two maps $\Hh \to \Q(v) \otimes R(\Gbf^{\Frob})$
\begin{enumerate}
\item $\Hh = K_0(\Ocal^{\DL, \mix, \unip}_{\Qlb}) \xrightarrow{\ch} K_0(\Rep^{\mix}_{\Qlb} \Gbf^{\Frob}) \xrightarrow{\gamma} \Q(v) \otimes R(\Gbf^{\Frob})$, we shall denote this map by $\ch$, 
\item $\tr : \Hh \to \Q(v) \otimes R(\Gbf^{\Frob})$ the map defined by $\tr(f) = \sum_{\widetilde{E}} \tr(f\Frob, \widetilde{E}(u)) \otimes R_{\tilde{E}}$. 
\end{enumerate}

\begin{thm}[\protect{\cite[Theorem 3.8]{LusztigBook}}]\label{thm:Lusztig38}
There is an equality of $\Z[v^{\pm 1}]$-linear maps $\Hh \to \Q(v) \otimes R(\Gbf^{\Frob})$ 
$$\tr = \ch \circ b.$$
(Recall that $b$ is defined in Remark \ref{rq:Defb}.)
\end{thm}

\begin{proof}
Tracing the definitions, the statement of \cite[Theorem 3.8]{LusztigBook} is precisely this equality evaluated at $\underline{H}_w$. Since these elements form a basis of $\Hh$, the theorem follows. 
\end{proof}

\begin{thm}[\protect{\cite[Proposition 6.9]{LusztigBook}}]\label{thm:Lusztig69}
There is an equality in $\Q(v) \otimes R(\Gbf^{\Frob})$,
$$d\ch(\underline{\tilde{H}}_{w,1}) = \pm\tr(\underline{\tilde{H}}_{w,1}).$$
\end{thm}

The signs are determined in \emph{loc. cit.}, we will not need them so we do not introduce them to avoid cumbersome notations.

Using the categorification Theorem \ref{lem:Categorification}, we then deduce the following corollary. 

\begin{corol}\label{cor:Duality}
There is an equality in $\Q(v) \otimes R(\Gbf^{\Frob})$ 
$$d(\ch([\IC_{w,1}^{\mix}])) = \pm\ch([T_{w,1}^{\mix}]).$$
\end{corol}

\begin{rque}
We have restricted ourselves to the unipotent case, as this suffices to prove the conjecture of Dudas–Malle. However it should be noted that, assuming Conjecture \ref{conj:CharacterMixedTilting}, a non-unipotent variant of Corollary \ref{cor:Duality} could be deduced. Indeed, Theorems \ref{thm:Lusztig38} and \ref{thm:Lusztig69} are both shown in \cite{LusztigBook} in the non-unipotent case. 
\end{rque}

\subsection{Eigenvalues of Frobenius}

We make a quick reminder on the theory of eigenvalues of Frobenius. 

\begin{thm}[\protect{\cite[Theorem 3.8.1]{LusztigBook}}]\label{thm:ExistenceRoot1}
Let $\rho \in \Irr_{\Qlb}^{\unip} \Gbf^{\Frob}$ be an irreducible unipotent representation. There exists a root of unity $\zeta_{\rho} \in \Qlb^{\times}$ satisfying the following property, for all $A \in \Ocal^{\DL, \unip, \mix}$, the eigenvalues of $\Frob^{\delta}$ on
the $\rho$-isotypic component of $\ch(A)$, are of the form $q^{\frac{i}{2}}\zeta$ for $i \in \Z$. 
\end{thm}

\begin{proof}
This exact statement is proved in \cite[Theorem 3.8.1]{LusztigBook} when $A = \Delta_{w,1}^{\mix}$ (and thus $\ch(A) = \RGamma_c(X(w), \Qlb)[\ell(w)]$). Since these objects generate the category $\Ocal^{\DL, \unip, \mix}$, the statement follows. 
\end{proof}

Let $\rho$ be as in Theorem \ref{thm:ExistenceRoot1}, and let $d(\rho) \in \Irr_{\Qlb}^{\unip} \Gbf^{\Frob}$ be the Alvis--Curtis dual of $\rho$ (written without signs). 

\begin{lem}\label{lem:DualityRootOf1}
The following equality holds 
$$\zeta_{d(\rho)} = \zeta_{\rho}.$$
\end{lem}

\begin{proof}
By \cite[3.33]{LusztigBook77} (see also \cite[Lemma 4.2]{GeckMalle}), the map 
$$\rho \mapsto \zeta_{\rho}$$ 
is constant on Harish-Chandra series. Since Harish-Chandra series are invariant under Alvis--Curtis duality, the statement follows.
\end{proof}

Recall (see \ref{sec:Notations}) that we denote by $\DD_c(\pt/\Gbf^{\Frob}, \Qlb)^{\Weil}$ the category of pairs $(\rho, \alpha)$ where $\rho \in \DD_c(\pt/\Gbf^{\Frob}, \Qlb)$ and $\alpha$ is an isomorphism $\Frob^{\delta,*}\rho \simeq \rho$. For $\lambda \in \Qlb$ and $(\rho, \alpha) \in \DD_c(\pt/\Gbf^{\Frob}, \Qlb)^{\Weil}$, we denote by 
$$\rho[\lambda]$$
the direct summand of $(\rho, \alpha)$ where $\alpha$ has generalized eigenvalue $\lambda$. For $\overline{\lambda} \in \Flb$, we denote by 
$$\rho[\overline{\lambda}] = \bigoplus_{\lambda} \rho[\lambda]$$ 
where the direct sum is indexed by all $\lambda \in \Zlb$ which reduce to $\overline{\lambda} \in \Flb$. 

\begin{lem}\label{lem:DualityEVFrob}
Let $A \in \Ocal^{\DL, \unip, \mix}$, then there is an equality in $K_0(\Rep_{\Qlb}(\Gbf^{\Frob}))$ for all $\lambda \in \Qlb^{\times}$,
$$d(\ch(A)[\lambda]) = (d(\ch(A))[\lambda].$$
\end{lem}

\begin{proof}
It is enough to prove the statement for a generating collection of sheaves in $\Ocal^{\DL, \unip, \mix}$, we will prove it for $\IC_w^{\mix}$. By Theorem \ref{thm:ExistenceRoot1}, we can assume that $\lambda = q^{\frac{i}{2}}\zeta$ for some root of unity $\zeta$. Since the equality of Corollary \ref{cor:Duality} is $\Q(v)$-linear, we have
\begin{equation}\label{eq:DualityCal1}
\sum_{\zeta'} d(\IC_w^{\mix}[q^{\frac{i}{2}}\zeta']) = \sum_{\zeta'} d(\IC_w^{\mix})[q^{\frac{i}{2}}\zeta'].
\end{equation}
Let $\rho \in \Irr_{\Qlb}^{\unip} \Gbf^{\Frob}$, applying $\langle \rho, - \rangle$ to \ref{eq:DualityCal1}, we have 
\begin{equation}\label{eq:DualityCal2}
\sum_{\zeta'} \langle \rho, d(\IC_w^{\mix}[q^{\frac{i}{2}}\zeta']) \rangle = \sum_{\zeta'} \langle \rho, d(\IC_w^{\mix})[q^{\frac{i}{2}}\zeta'] \rangle. 
\end{equation}
By Theorem \ref{thm:ExistenceRoot1}, the LHS is 
\begin{align*}
\sum_{\zeta'} \langle \rho, d(\IC_w^{\mix}[q^{\frac{i}{2}}\zeta']) \rangle &= \sum_{\zeta'} \langle d(\rho), \IC_w^{\mix}[q^{\frac{i}{2}}\zeta'] \rangle \\
&= \langle d(\rho), \IC_w^{\mix}[q^{\frac{i}{2}}\zeta_{d(\rho)}] \rangle \\
&= \langle \rho, d(\IC_w^{\mix}[q^{\frac{i}{2}}\zeta_{d(\rho)}]) \rangle. 
\end{align*}
Similarly, the RHS evaluates to
$$\sum_{\zeta'} \langle \rho, d(\IC_w^{\mix})[q^{\frac{i}{2}}\zeta'] \rangle = \langle \rho,  d(\IC_w^{\mix})[q^{\frac{i}{2}}\zeta_{\rho}] \rangle.$$
It thus follows from Lemma \ref{lem:DualityRootOf1} that 
$$\langle \rho, d(\IC_w^{\mix}[q^{\frac{i}{2}}\zeta]) \rangle = \langle \rho, d(\IC_w^{\mix})[q^{\frac{i}{2}}\zeta] \rangle$$
for all $\zeta$ and the lemma follows. 
\end{proof}

\begin{corol}\label{cor:DualityEVFrob}
Let $A \in \Ocal^{\DL, \unip, \mix}$, for all $\overline{\lambda} \in \Flb^{\times}$, there is an equality in $K_0(\Rep_{\Qlb}(\Gbf^{\Frob}))$ 
$$d(\ch(A)[\overline{\lambda}]) = (d(\ch(A))[\overline{\lambda}]).$$
\end{corol}

\begin{proof}
This is simply a matter of grouping terms and applying Lemma \ref{lem:DualityEVFrob}. 
\end{proof}

\subsection{A conjecture of Dudas--Malle}

The goal of this section is to provide an answer to a conjecture of Dudas and Malle \cite[Conjecture 1.2]{DudasMalle}.
The following hypothesis is in force in this section :  

\begin{center}
The Kazhdan--Lusztig basis and the $\ell$-Kazhdan--Lusztig basis coincide. 
\end{center}

By Lemma \ref{lem:IncomposabilityLargeL} this is guaranteed whenever $\ell$ is large enough (depending only on $\Gbf$ and not on the Frobenius). 

Let $(w, \chi)$ be a pair with $w \in \Wbf$ and $\chi$ a $\Qlb$-character of $\Tbf^{w\Frob}$. The object $\IC_{w,\chi}^{\mix}$ is equipped with a canonical Weil-structure, i.e., an isomorphism
$$\Frob^{\delta,*}\IC_{w,\chi}^{\mix} \simeq \IC_{w,\chi}^{\mix}.$$
By functoriality and compatibility with the Frobenius, the object $\ch(\IC_{w,\chi}^{\mix})$ is also equipped with an isomorphism 
$$\Frob^{\delta,*}\ch(\IC_{w,\chi}^{\mix}) \simeq \ch(\IC_{w,\chi}^{\mix}).$$

\begin{thm}[\protect{\cite[Conjecture 1.2]{DudasMalle}}]\label{thm:ConjDudasMalle}
For all $\overline{\lambda} \in \Flb$, the character 
$$d(\ch(\IC_{w,1}^{\mix}[\overline{\lambda}]))$$
is the unipotent part of a character of a projective $\Zlb$-representation. 
\end{thm}

\begin{proof}
Consider the object $T_{w, 1}^{\Zlb}$. We have 
$$T_{w, 1}^{\Zlb}[\frac{1}{\ell}] = T_{w,1}^{\Qlb} \oplus \bigoplus_{\chi_{\ell} \neq 1} T_{w, \chi_{\ell}}^{\Qlb} \oplus \bigoplus_{v < w, \chi_{\ell} \neq 1} T_{v,\chi_{\ell}}^{\Qlb, \oplus {n_{v,w,\chi_{\ell}}}}$$
by Lemmas \ref{lem:Decomp} and \ref{lem:IncomposabilityLargeL}. 
We equip $T_{w, 1}^{\Zlb}[\frac{1}{\ell}]$ with a Weil structure such that there is an isomorphism of sheaves with Weil structures 
$$T_{w, 1}^{\Zlb}[\frac{1}{\ell}] = T_{w,1}^{\Qlb, \mix} \oplus \bigoplus_{\chi_{\ell} \neq 1} T_{w, \chi_{\ell}}^{\Qlb, \mix} \oplus \bigoplus_{v < w, \chi_{\ell} \neq 1} T_{v,\chi_{\ell}}^{\Qlb,\mix, \oplus {n_{v,w,\chi_{\ell}}}}.$$
Applying $\ch(-)[\overline{\lambda}]$ we get an isomorphism
$$\ch(T_{w, 1}^{\Zlb}[\frac{1}{\ell}])[\overline{\lambda}] = \ch(T_{w,1}^{\Qlb, \mix})[\overline{\lambda}] \oplus \bigoplus_{\chi_{\ell} \neq 1} \ch(T_{w, \chi_{\ell}}^{\Qlb, \mix})[\overline{\lambda}] \oplus \bigoplus_{v < w, \chi_{\ell} \neq 1} \ch(T_{v,\chi_{\ell}}^{\Qlb,\mix, \oplus {n_{v,w,\chi_{\ell}}}})[\overline{\lambda}].$$
The unipotent part of this representation is $\ch(T_{w,1}^{\Qlb, \mix})[\overline{\lambda}]$. Since mixed tilting sheaves are filtered, as sheaves with Weil structure, by mixed standard sheaves, it follows that the only eigenvalues of Frobenius appearing in $\ch(T_{w,\chi}^{\Qlb, \mix})$ are eigenvalues of Frobenius appearing on the cohomology of Deligne--Lusztig varieties. By Theorem \ref{thm:ExistenceRoot1}, these eigenvalues of Frobenius are of the form $q^{\frac{i}{2}}\zeta$ where $i \in \Z$ and $\zeta \in \Qlb^{\times}$ are roots of unity. It follows that the Weil structure on $\ch(T_{w,1}^{\Zlb}[\frac{1}{\ell}])$, which is an isomorphism
$$\Frob^{\delta, *}\ch(T_{w,1}^{\Zlb}[\frac{1}{\ell}])\simeq \ch(T_{w,1}^{\Zlb}[\frac{1}{\ell}])$$
descends to $\Zlb$ as all its eigenvalues lie in $\Zlb$. Hence we have a well-defined $\Zlb$-linear isomorphism 
$$\Frob^{\delta,*}\ch(T_{w,1}^{\Zlb}) \simeq \ch(T_{w,1}^{\Zlb}).$$
There is therefore a well-defined direct summand of $\ch(T_{w,1}^{\Zlb})$ which we denote by $\ch(T_{w,1}^{\Zlb})[\overline{\lambda}]$ over $\Zlb$ and which is characterized by the property that 
$$\ch(T_{w,1}^{\Zlb})[\overline{\lambda}][\frac{1}{\ell}] = \ch(T_{w,1}^{\Zlb})[\frac{1}{\ell}][\overline{\lambda}].$$
Since $\ch(T_{w,1}^{\Zlb})[\overline{\lambda}]$ is a direct summand of $\ch(T_{w,1}^{\Zlb})$, by Theorem \ref{thm:TiltingRepsAreProjective}, this is a projective representation of $\Gbf^{\Frob}$ over $\Zlb$. To prove the theorem, it is enough to show that, after taking characters, we have 
$$d(\ch(T_{w, 1}^{\Qlb, \mix}[\overline{\lambda}])) = \ch(\IC_{w,1}^{\Qlb, \mix})[\overline{\lambda}].$$
This is exactly Corollary \ref{cor:DualityEVFrob}.
\end{proof}

\bibliographystyle{alpha}
\bibliography{biblio}

\end{document}